\documentclass{amsart}%
\usepackage{eurosym}
\usepackage{amsmath}
\usepackage{amsfonts}
\usepackage{amssymb}
\usepackage{graphicx}
\usepackage{color}
\usepackage{bm}
\usepackage{caption}
\usepackage{subcaption}
\usepackage{verbatim}
\usepackage{hyperref}
\usepackage{bmpsize}
\usepackage{amsxtra,comment,graphicx,psfrag}
\usepackage{epstopdf}
\usepackage{bm,mathrsfs}
\usepackage{mathtools}
\usepackage{booktabs,siunitx}
\usepackage{multirow}
\usepackage{makecell}
\usepackage{amssymb,amsfonts,amsmath}
\usepackage{color}
\usepackage[margin=1.0in]{geometry}
\usepackage{epstopdf}
\usepackage{enumerate,enumitem}
\usepackage{cprotect}
\usepackage{algorithm,algpseudocode}
\usepackage{graphics}
\usepackage{color,cancel}
\usepackage{amsmath,amsfonts,amscd,amssymb,bm}
\usepackage{mathrsfs,cite}
\usepackage{fancyhdr,fancybox}
\usepackage{listings}
\usepackage{url}
\usepackage{enumerate}
\usepackage{marginnote}
\usepackage{array}
\usepackage{comment}
\usepackage{epstopdf}
\usepackage{placeins}
\usepackage{graphicx}%

\providecommand{\U}[1]{\protect\rule{.1in}{.1in}}
\allowdisplaybreaks
\algnewcommand\algorithmicinput{\textbf{Input:}}
\algnewcommand\Input{\item[\algorithmicinput]}
\algnewcommand\algorithmicoutput{\textbf{Output:}}
\algnewcommand\Output{\item[\algorithmicoutput]}
\definecolor{blue}{rgb}{0,0.0,0.9}
\definecolor{red}{rgb}{0.8,0.0,0}

\theoremstyle{plain}
\newtheorem{theorem}{Theorem}[section]
\newtheorem{lemma}[theorem]{Lemma}

\theoremstyle{definition}
\newtheorem{definition}[theorem]{Definition}

\theoremstyle{remark}

\newcommand{\be}{\begin{equation}}
\newcommand{\ee}{\end{equation}}
\newcommand{\bea}{\begin{eqnarray}}
\newcommand{\eea}{\end{eqnarray}}
\newcommand{\beas}{\begin{eqnarray*}}
\newcommand{\eeas}{\end{eqnarray*}}

\includecomment{confidential}
\definecolor{db}{rgb}{0.0470,0,0.5294}
\definecolor{dg}{rgb}{0.0,0.392,0.0}
\definecolor{firebrick}{rgb}{0.698,0.133,0.133}
\definecolor{bl}{rgb}{0.0,0.0,0.0}
\definecolor{linen}{rgb}{0.980,0.941,0.902}
\definecolor{ivory}{rgb}{1.0,1.0,0.941}
\definecolor{aliceblue}{rgb}{0.941,0.973,1.0}
\definecolor{beige}{rgb}{0.961,0.961,0.863}
\definecolor{tan}{rgb}{0.824,0.706,0.549}
\definecolor{lightsteelblue}{rgb}{0.690,0.769,0.871}
\definecolor{paleturquoise}{rgb}{0.686,0.933,0.933}
\definecolor{lightblue}{rgb}{0.678,0.847,0.902}
\definecolor{skyblue}{rgb}{0.529,0.808,0.922}
\definecolor{darkblue}{rgb}{0.00,0.00,0.55}
\definecolor{palegoldenrod}{rgb}{0.933,0.910,0.667}
\definecolor{lightgoldenrod}{rgb}{0.933,0.867,0.510}
\definecolor{lightyellow}{rgb}{1.0,1.0,0.878}
\definecolor{yellow}{rgb}{1.0,1.0,0.0}
\definecolor{lightyellow1}{rgb}{1.0,1.0,0.878}
\definecolor{lemonchiffon}{rgb}{1.0,0.980,0.804}
\definecolor{myyellow}{rgb}{1,1,.9}
\definecolor{darkgreen}{rgb}{0.0,0.392,0.0}
\definecolor{darkviolet}{rgb}{0.580,0.0,0.827}
\definecolor{lightsalmon}{rgb}{1.0,0.627,0.478}
\definecolor{orange}{rgb}{1.0,0.647,0.0}
\newcommand{\vertiii}[1]{{\left\vert\kern-0.25ex\left\vert\kern-0.25ex\left\vert #1
\right\vert\kern-0.25ex\right\vert\kern-0.25ex\right\vert}}
\begin{document}

\title{Data Assimilation with Sparse Observations}

\author{Nan Jiang}
\thanks{Department of Mathematics, University of Florida, Gainesville, FL 32611, \texttt{jiangn@ufl.edu}. The first author was partially supported by the US National Science Foundation grant DMS-2143331.}

\author{William Layton}
\thanks{Department of Mathematics, University of Pittsburgh, Pittsburgh, PA 15260, \texttt{wjl@pitt.edu}. The second author was partially supported by the US National Science Foundation grant DMS-2410893.}

\author{Nanda Nechingal Raghunathan}
\thanks{Department of Mathematics, University of Pittsburgh, Pittsburgh, PA 15260, \texttt{nan158@pitt.edu}. The third author was partially supported by the US National Science Foundation grant DMS-2410893.}

\author{Troy Yang}
\thanks{Department of Mathematics, University of Pittsburgh, Pittsburgh, PA 15260, \texttt{try21@pitt.edu}. The fourth author was partially supported by the US National Science Foundation grant DMS-2410893.}
\subjclass[2020]{Primary 65M12; Secondary 65M60}
\keywords{Nudging, CDA, data assimilation, predictability horizon}

\begin{abstract}
Data assimilation by nudging (also called CDA) yields exponentially decaying
errors and an infinite predictability horizon if the method parameter is large
enough and the observations are frequent enough in time and dense enough in
space. We consider the complementary case of moderate parameters and sparse
and infrequent observations. We prove that assimilation with any data and any
(positive) parameter strictly decreases errors and strictly increases the (now
finite) predictability horizon.

\end{abstract}
\maketitle
\bigskip

\section{Introduction}

Data assimilation by nudging/continuous data assimilation (CDA), originating
in 1964, \cite{L64} and with many connections \cite{ref2}, is supported by a
stunning theoretical result from 2014 in \cite{ref1} that has been elaborated
in many directions, \cite{BM17, LRZ19, ZRSI19, RZ21, GLRVZ21, LHRV23, HRV24,
GLNR24, DLR25,DR22, ADR24,Titi16}. These results assume that the data is
frequent and dense enough, and that the nudging parameter is large enough.
This report analyzes CDA when these assumptions are not satisfied. The case of
sparse, infrequent data and moderate nudging parameter is realistic, and of
practical importance.

In CDA, velocity data/observations, denoted $I_{H}u$, occur on a network with
characteristic spacing ``$H$". These are observations of the exact solution of
the Navier-Stokes equations%
\begin{equation}
u_{t}+u\cdot\nabla u-\nu\Delta u+\nabla p=f(x,t)\qquad\text{and}\qquad
\nabla\cdot u=0, \label{eq:NSE}%
\end{equation}
with appropriate boundary and initial conditions.\ \ The nudged, CDA
approximation chooses parameter $\chi>>1$ and solves%
\begin{equation}
v_{t}+v\cdot\nabla v-\nu\Delta v+\nabla q-\chi I_{H}(u-v)=f(x,t)\qquad
\text{and}\qquad\nabla\cdot v=0. \label{eq:CDA}%
\end{equation}
Herein we assume the system has the same boundary conditions but different
initial conditions. The observation operator $I_{H}$ satisfies stability and
approximability conditions, such as for a projection into a coarse mesh,
piecewise polynomial space. In 2014 Azouani, Olsen and Titi \cite{ref1} proved
that with observations dense enough ($H\leq C(\mathcal{R}e^{-1})<<1$) and
nudging parameter large enough ($\chi\geq C(\mathcal{R}e)>>1$), the error in
the CDA approximation, $||u(t)-v(t)||_{L^{2}}$ $\rightarrow0$ exponentially,
i.e., the CDA predictability horizon $T_{P}=\infty$. Unfortunately, for many
flows the \textquotedblleft dense enough \& large enough\textquotedblright%
\ conditions are unrealistic (see \cite{ref5} for specific details). We
address herein the resulting critical and unanswered question:

\begin{center}
\emph{What is the effect of nudging/CDA with moderate parameter and sparse
observations?}
\end{center}

To date and to our knowledge, the only partial answer proven in \cite{ref4} is
that for a specific discrete-in-time, 2-step, modular data assimilation
algorithm, incorporating any observation strictly decreases errors and
strictly increases the predictability horizon $T_{P}$. We prove herein an
extension of this result to semi-implicit, discrete time realization of
standard nudging/CDA. The result herein has one key feature we conjecture is
improvable. It is for a standard semi-implicit (meaning explicit treatment of
$v\cdot\nabla v$\ and implicit treatment of other terms) time discretization.
The same result for the continuum case of (\ref{eq:CDA}) is an open problem.

Section \ref{sec2} gives a precise formulation of the above question and
presents the result proven in Section \ref{sec3} for the simplest context of
discrete time and continuous space. Section \ref{sec4} proves the basic result
is retained after both space and time discretization. Numerical experiments
are presented in Section \ref{sec5}. Conclusions are in Section \ref{sec6}.

\section{Formulating the Result}

\label{sec2}

The (sufficiently smooth) true solution sought is a solution of the
Navier--Stokes equations (NSE) in a bounded regular domain $\Omega$ in 2D or
3D with body force $f(x,t)$ and kinematic viscosity $\nu$ and subject to
no-slip boundary conditions
\[
u_{t} + u \cdot\nabla u - \nu\Delta u + \nabla p = f(x,t) \qquad\text{and}
\qquad\nabla\cdot u = 0.
\]

We consider discrete time approximations with inexact initial conditions. A
standard semi-implicit method is as follows. Let $\Delta t=$ timestep, $n=$
step number, $t_{n} = n \Delta t$, $v^{n}(x)$ \& $u^{n}(x) \simeq u(x,t_{n})$.
Suppressing the space discretization, the non-nudged, semi-implicit method
considered is: $\nabla\cdot w^{n+1}=0$ and
\[
\frac{w^{n+1}-w^{n}}{\Delta t} + w^{n} \cdot\nabla w^{n} - \nu\Delta w^{n+1} +
\nabla q^{n+1} = f(x,t),
\]
where in general $w^{0} \neq u(0)$. See \cite{ref6} for its numerical analysis.

For the nudged/CDA approximation let $\chi=$ positive nudging parameter, $H=$
observational network spacing and $I_{H}=$ operator that takes observation
data and produces a velocity on $\Omega$. The nudged approximation is:
\[
v^{0} = w^{0} \neq u(0), \qquad\nabla\cdot v^{n+1}=0
\]
and
\[
\frac{v^{n+1}-v^{n}}{\Delta t} + v^{n} \cdot\nabla v^{n} - \nu\Delta v^{n+1} -
\chi I_{H} \bigl(u(t_{n+1}) - v^{n+1}\bigr) + \nabla r^{n+1} = f(x,t).
\]

The usual $L^{2}(\Omega)$ norm and inner product are denoted $\Vert\cdot\Vert
$, $(\cdot,\cdot)$. To formulate the result, we suppose we have sparse ($H$
not small) data at only 1 timestep, $t_{n+1}$, and compare the errors of the
nudged step with the error of the standard, non-nudged step. In other words,
we assume that no data is available at $t^{n}$, so that $w^{n} = v^{n}$. At
$t^{n+1}$, this may no longer hold, since data is available at that time step.
We prove that the effect of the assimilation term is to strictly decrease the
error, measured in the norm
\[
|||\phi|||^{2}:=\Vert\phi\Vert^{2}+\Delta t\nu\Vert\nabla\phi\Vert^{2}.
\]
This decrease means compared to performing the same step with the same
starting velocity $v^{n}$, body force and without the nudging term.

\bigskip

\textbf{Theorem 2.1.} Let $v^{n}=w^{n}$. Suppose the operator $I_{H}$
satisfies
\[
(I_{H}\phi,\phi)\geq0\qquad\text{for all }\phi.
\]
Then, for any positive $H$ and $\chi$, the error is strictly decreased by
CDA:
\[
|||u(t_{n+1})-v^{n+1}|||<|||u(t_{n+1})-w^{n+1}|||
\]
unless $v^{n+1}=w^{n+1}$ and consequently
\[
I_{H}(u(t_{n+1})-v^{n+1})\equiv0.
\]
Specifically,
\begin{equation}
\begin{aligned} &|||u(t_{n+1})-v^{n+1}|||^2 + 2\Delta t\chi\bigl(I_H(u(t_{n+1})-v^{n+1}), (u(t_{n+1})-v^{n+1})\bigr) \\ &\qquad + \|v^{n+1}-w^{n+1}\|^2 = |||u(t_{n+1})-w^{n+1}|||^2. \end{aligned} \tag{2.1}%
\end{equation}

Since CDA/nudging strictly decreases errors, for any reasonable definition of
the predictability horizon based on errors at $t_{n}$ and $t_{n+1}$, it
strictly increases $T_{P}$. For a precise mathematical elaboration of this
condition (based on finite time Lyapunov exponents \cite{ref3}) see
\cite{ref4}.

\section{Proof of Theorem 2.1}

\label{sec3}
\begin{proof}
The proof requires error equations for
\[
e^{n+1}=u(t_{n+1})-w^{n+1} \qquad\text{and} \qquad\varepsilon^{n+1}%
=u(t_{n+1})-v^{n+1}.
\]
To obtain them, integrate the NSE over $t_{n}<t<t_{n+1}$ and rearrange. This
gives
\[
u(t_{n+1})-u(t_{n}) + \Delta t\bigl(u(t_{n})\cdot\nabla u(t_{n})\bigr) -
\nu\Delta u(t_{n+1}) + \nabla p(t_{n+1})\bigr) = \Delta t f^{n+1}+\tau
\]
where
\[
\begin{aligned} \tau :=\;& \Delta t\bigl( u(t_n)\cdot \nabla u(t_n) -\nu\Delta u(t_{n+1}) +\nabla p(t_{n+1}) \bigr) \\ &- \int_{t_n}^{t_{n+1}} u(t)\cdot \nabla u(t) -\nu\Delta u(t) +\nabla p(t)\,dt . \end{aligned}
\]
By subtraction, $e^{n+1},\varepsilon^{n+1}$ satisfy:
\[
\nabla\cdot e^{n+1} = \nabla\cdot\varepsilon^{n+1} = 0,
\]
\begin{equation}
e^{n+1}-e^{n} +\Delta t\bigl(u(t_{n})\cdot\nabla u(t_{n})-w^{n}\cdot\nabla
w^{n}\bigr) -\Delta t\nu\Delta e^{n+1} +\Delta t\nabla\bigl(p(t_{n+1}%
)-q^{n+1}\bigr) =\tau, \tag{3.1}%
\end{equation}
and
\[
\varepsilon^{n+1}-\varepsilon^{n} +\Delta t\bigl(u(t_{n})\cdot\nabla
u(t_{n})-v^{n}\cdot\nabla v^{n}\bigr) -\Delta t\nu\Delta\varepsilon^{n+1}
+\Delta t\chi I_{H} \varepsilon^{n+1} +\Delta t\nabla\bigl(p(t_{n+1}%
)-r^{n+1}\bigr) =\tau.
\]
By assumption, $v^{n}=w^{n}$ so $e^{n}=\varepsilon^{n}$. Equate the LHS of the
two (3.1) using $\tau=\tau$ and cancel like terms. This gives
\[
\varepsilon^{n+1} -\Delta t\nu\Delta\varepsilon^{n+1} +\Delta t\chi
I_{H}\varepsilon^{n+1} +\Delta t\nabla\bigl(p(t_{n+1})-r^{n+1}\bigr) = e^{n+1}
-\Delta t\nu\Delta e^{n+1} +\Delta t\nabla\bigl(p(t_{n+1})-q^{n+1}\bigr).
\]
Take the $L^{2}$ inner product with $\varepsilon^{n+1}$, integrate by parts
and use
\[
\nabla\cdot e^{n+1} = \nabla\cdot\varepsilon^{n+1} = 0
\]
to make the pressure error term drop out. This yields
\begin{equation}
\|\varepsilon^{n+1}\|^{2} +\Delta t\nu\|\nabla\varepsilon^{n+1}\|^{2} +\Delta
t\chi(I_{H}\varepsilon^{n+1},\varepsilon^{n+1}) = (e^{n+1},\varepsilon^{n+1})
+\Delta t\nu(\nabla e^{n+1},\nabla\varepsilon^{n+1}). \tag{3.2}%
\end{equation}
Drop the (now unnecessary) superscript ``$n+1$'' for clarity. The polarization
identity, using the $|||\cdot|||$ notation, for the two terms on the RHS
gives
\[
(e,\varepsilon) +\Delta t\nu(\nabla e,\nabla\varepsilon) = \frac12 |||e|||^{2}
+\frac12 |||\varepsilon|||^{2} -\frac12 |||e-\varepsilon|||^{2}.
\]
Insert this in (3.2), simplifying, multiply by 2 and use
\[
e-\varepsilon=v-w.
\]
This yields the following and thus completes the proof:
\[
|||\varepsilon|||^{2} + 2\Delta t\chi(I_{H}\varepsilon,\varepsilon) +
|||v-w|||^{2} = |||e|||^{2}.
\]
\end{proof}

\section{Fully discrete case}

\label{sec4}

In this section we provide the corresponding result for the fully discrete
schemes. The $L^{p}(\Omega)$ norms and the Sobolev $W^{k}_{p}(\Omega)$ norms
are denoted $\|\cdot\|_{L^{p}}$ and $\|\cdot\|_{W_{p}^{k}}$ respectively.
$H^{k}(\Omega)$ is the Sobolev space $W_{2}^{k}(\Omega)$, with norm
$\|\cdot\|_{k}$.  Let $t_{n}=n\Delta t,n=0,1,2,...,N_{T},$ and $T=N_{T}\Delta t$. We assume the
regularity $ u\in L^{\infty }\left(0,T;H^{k+1}(\Omega )\right)$ and define the following norm
\begin{equation*}
\Vert u\Vert _{\infty ,k}\text{ }:=\sup_{0\leq t\leq T}\Vert u(\cdot, t )\Vert _{k}.
\end{equation*}%

\noindent Let $X$ be the velocity space and $Q$ be the pressure space:
\[
X\text{ }:=(H_{0}^{1}(\Omega))^{d}, \ \text{ }Q\text{ }:=L_{0}^{2}(\Omega).
\]

\noindent For $v\in X$, the Poincar$\acute{e}$ inequality holds,
\[
\Vert v\Vert\leq C_{P}\Vert\nabla v\Vert.
\]

\noindent The space of divergence free functions is given by
\[
V\text{ }:=\{v\in X\text{ }:(\nabla\cdot v,q)=0\text{ , }\forall q\in Q\}.
\]
Define the usual explicitly skew symmetric trilinear form%
\[
b^{\ast}\left(  u,v,w\right)  :=\frac{1}{2}\left(  u\cdot\nabla v,w\right)
-\frac{1}{2}\left(  u\cdot\nabla w,v\right)  .
\]
The same results hold for the alternate, skew symmetric trilinear form
\begin{equation}
b^{\ast}\left(  u,v,w\right)  :=\left(  u\cdot\nabla v+1/2(\nabla\cdot
u)v,w\right)  . \label{eq:bstar}%
\end{equation}

We base our analysis on the standard finite element method (FEM) for the
spatial discretization. The results also extend to many other space
discretizations, including adding various stabilization terms. \ Let
$X_{h}\subset X,Q_{h}\subset Q$ denote conforming velocity, pressure finite
element spaces based on an edge to edge triangulations of $\Omega$ with
maximum triangle diameter $h$. The velocity-pressure FEM spaces $(X_{h}%
,Q_{h})$ are assumed to satisfy the usual discrete inf-sup/$LBB^{h}$
condition, see \cite{G89, Layton08}, for stability of the discrete pressure:
\[
\inf_{q_{h}\in Q_{h}}\sup_{v_{h}\in X_{h}}\dfrac{(q_{h},\nabla\cdot v_{h}%
)}{\Vert q_{h}\Vert\Vert\nabla v_{h}\Vert}\geq\beta^{LBB} >0,
\]
where $\beta^{LBB}$ is independent of $h$. The velocity-pressure FEM spaces
are also assumed to satisfy the following approximation properties:

\begin{gather}
\inf_{v_{h}\in X_{h}}\Vert\nabla\left(  u-v_{h}\right)
\Vert\leq Ch^{k} \Vert u\Vert_{k+1}, \qquad u\in H^{k+1}(\Omega)^{d} \cap X ,\label{estimate}\\
\inf_{q_{h}\in Q_{h}}\Vert  q-q_{h}
\Vert\leq Ch^{k} \Vert p\Vert_{k}, \qquad q\in H^{k}(\Omega)\cap Q.
\end{gather}

\noindent Taylor-Hood elements $(P^{k+1}, P^{k})$,  \cite{G89}, are one commonly used choice of
such velocity-pressure finite element spaces. 

The observation spacing is denoted by $H$. The interpolation operator $I_{H}$
takes observations on a coarse (observational) mesh and interpolates them onto
the fine mesh used for the numerical simulation. If $X_{H}$, $X_{h}$ denote
these coarse (observation) and fine (simulation) subspaces of $X$, we assume
$I_{H}:X\rightarrow X_{h}$ satisfies
\begin{align}
\left\Vert I_{H}(w)-w\right\Vert  &  \leq\hat{C}H\Vert\nabla w\Vert
,\label{bound}\\
\left\Vert I_{H}(w)\right\Vert  &  \leq\breve{C}\Vert w\Vert, \label{bound1}%
\end{align}
for any $w\in H^{1}(\Omega)$ and with $\hat{C}$ and $\breve{C}$ independent of
$H$. Examples of such interpolations include the $L^{2}$ projection onto
piecewise constants and the Scott-Zhang interpolation. The FE implementation
of the nudging term in the fully discrete algorithm appends $\chi\left(
I_{H}(u_{h}),I_{H}(v_{h})\right)  $.

The discretely divergence free subspace of $X_{h}$ is
\[
V_{h}\text{ }:=\left\{  v_{h}\in X_{h}:(\nabla\cdot v_{h},q_{h})=0\text{ , }
\forall q_{h}\in Q_{h}\right\}  .
\]

\begin{definition}
[The Modified H1 Projection] Given $u(t)\in X$, we define the projection $P^{h}
u(t)\in V_{h}$ is defined to be the solution of
\begin{align}
\label{11111a} &  \left(  P^{h} u(t), v_{h}\right)  +\Delta t\nu\left(  \nabla
P^{h} u(t),\nabla v_{h}\right)  = \left(  u(t), v_{h}\right)  +\Delta
t\nu\left(  \nabla u(t),\nabla v_{h}\right)  ,\quad\forall v_{h}\in V_{h}.
\end{align}

\end{definition}

We have the following approximation property for the projection.

\begin{lemma}%
\begin{gather}
||| P^{h} u(t)-u(t)|||\leq2 \sqrt{(C_{P}^{2} +\Delta t\nu)} \inf_{\phi_{h}\in
X_{h}}\Vert\nabla\left(  \phi_{h}-u(t)\right)  \Vert\leq C \sqrt{(C_{P}^{2}
+\Delta t\nu)} \left(  1+\frac{1}{\beta^{LBB}}\right)  h^{k} \Vert
u\Vert_{\infty, k+1}.
\end{gather}

\end{lemma}

\begin{proof}
Let $\phi_{h}$ be any function in $V_{h}$. We rewrite%

\begin{gather}
P^{h} u(t)-u(t)=\left(  \phi_{h}-u(t)\right)  -\left(  \phi_{h}-P^{h}
u(t)\right)  =\alpha-\beta_{h}.
\end{gather}

\noindent Then \eqref{11111a} can be written as
\begin{align}
\label{11111b} &  \left(  \alpha, v_{h}\right)  +\Delta t\nu\left(
\nabla\alpha,\nabla v_{h}\right)  = \left(  \beta_{h}, v_{h}\right)  +\Delta
t\nu\left(  \nabla\beta_{h},\nabla v_{h}\right)  ,\quad\forall v_{h}\in V_{h}.
\end{align}

\noindent Setting $v_{h}=\beta_{h}$ we have%

\begin{align}
\label{11111c} &  |||\beta_{h}|||^{2} =\Vert\beta_{h}\Vert^{2} +\Delta
t\nu\Vert\nabla\beta_{h}\Vert^{2} =\left(  \alpha, \beta_{h}\right)  +\Delta
t\nu\left(  \nabla\alpha,\nabla\beta_{h}\right)  \leq\Vert\alpha\Vert
\Vert\beta_{h}\Vert+\Delta t\nu\Vert\nabla\alpha\Vert\Vert\nabla\beta_{h}%
\Vert\leq|||\alpha|||\,|||\beta_{h}|||,
\end{align}

\noindent and thus%

\begin{align}
\label{11111d} &  |||\beta_{h}|||\leq|||\alpha|||=\sqrt{\Vert\alpha\Vert^{2}
+\Delta t\nu\Vert\nabla\alpha\Vert^{2}}\leq\sqrt{(C_{P}^{2} +\Delta t\nu)}
\Vert\nabla\alpha\Vert.
\end{align}

\noindent So%

\begin{gather}
||| P^{h} u(t)-u(t)|||=|||\alpha-\beta_{h} ||| \leq|||\alpha|||+|||\beta_{h}
|||\leq2 \sqrt{(C_{P}^{2} +\Delta t\nu)} \Vert\nabla\left(  \phi
_{h}-u(t)\right)  \Vert, \qquad\forall\phi_{h}\in V_{h}.
\end{gather}

\noindent Taking the infimum over $\phi_{h}\in V_{h}$ and applying
\eqref{estimate} yields%

\begin{gather}
||| P^{h} u(t)-u(t)||| \leq2 \sqrt{(C_{P}^{2} +\Delta t\nu)} \inf_{\phi_{h}\in
V_{h}}\Vert\nabla\left(  \phi_{h}-u(t)\right)  \Vert\nonumber\\
\leq2 \sqrt{(C_{P}^{2} +\Delta t\nu)} \left(  1+\frac{1}{\beta^{LBB}}\right)
\inf_{\phi_{h}\in X_{h}}\Vert\nabla\left(  \phi_{h}-u(t)\right)  \Vert\\
\leq C \sqrt{(C_{P}^{2} +\Delta t\nu)} \left(  1+\frac{1}{\beta^{LBB}}\right)
h^{k} \Vert u\Vert_{\infty, k+1}.\nonumber
\end{gather}

\noindent This concludes the proof.
\end{proof}

\subsection{Fully discrete Methods}

We now present the fully discrete methods for the nudging and non-nudging
cases respectively. The nonlinear terms are treated with fully explicit, first
order approximations. However, the proof works for any higher order
approximations (for the nonlinear terms) without any modifications.

\noindent\textbf{Nudging Method: }\noindent Given $v_{h}^{n}$, find
$v_{h}^{n+1}\in X_{h}$, $\lambda_{h}^{n+1}\in Q_{h}$ satisfying
\begin{align}
&  \left(\frac{v_{h}^{n+1}-v_{h}^{n}}{\Delta t},\phi_{h}\right)+b^{\ast}(v_{h}^{n}%
,v_{h}^{n},\phi_{h})-(\lambda_{h}^{n+1},\nabla\cdot\phi_{h}%
)\label{fullydiscrete}\\
&  +\chi(I_{H}(v_{h}^{n+1}-u^{n+1}),I_{H}\phi_{h})+\nu(\nabla v_{h}%
^{n+1},\nabla\phi_{h})=(f^{n+1},\phi_{h})\text{, }\qquad\forall\phi_{h}\in
X_{h},\nonumber\\
&  (\nabla\cdot v_{h}^{n+1},q_{h})=0,\qquad\forall q_{h}\in Q_{h}.\nonumber
\end{align}

\noindent which is equivalent to: \textit{find} $v_{h}^{n+1}\in V_{h}$ such that%

\begin{align}
&  \left(\frac{v_{h}^{n+1}-v_{h}^{n}}{\Delta t},\phi_{h}\right)+b^{\ast}(v_{h}^{n}%
,v_{h}^{n},\phi_{h})\label{fullydiscrete0}\\
&  +\chi(I_{H}(v_{h}^{n+1}-u^{n+1}),I_{H}\phi_{h})+\nu(\nabla v_{h}%
^{n+1},\nabla\phi_{h})=(f^{n+1},\phi_{h})\text{, }\qquad\forall\phi_{h}\in
V_{h}.\nonumber
\end{align}

\noindent\textbf{Method without Nudging:} \noindent Given $w_{h}^{n}$, find
$w_{h}^{n+1}\in X_{h}$, $\tilde{\lambda} _{h}^{n+1}\in Q_{h}$ satisfying
\begin{align}
&  \left(\frac{w_{h}^{n+1}-w_{h}^{n}}{\Delta t},\phi_{h}\right)+b^{\ast}(w_{h}^{n}%
,w_{h}^{n},\phi_{h})-(\tilde{\lambda} _{h}^{n+1},\nabla\cdot\phi
_{h})\label{fullydiscrete1}\\
&  +\nu(\nabla w_{h}^{n+1},\nabla\phi_{h})=(f^{n+1},\phi_{h})\text{, }%
\qquad\forall\phi_{h}\in X_{h},\nonumber\\
&  (\nabla\cdot w_{h}^{n+1},q_{h})=0,\qquad\forall q_{h}\in Q_{h},\nonumber
\end{align}

\noindent which is equivalent to: \textit{find} $w_{h}^{n+1}\in V_{h}$ such that%

\begin{align}
&  \left(\frac{w_{h}^{n+1}-w_{h}^{n}}{\Delta t},\phi_{h}\right)+b^{\ast}(w_{h}^{n}%
,w_{h}^{n},\phi_{h}) +\nu(\nabla w_{h}^{n+1},\nabla\phi_{h})=(f^{n+1},\phi
_{h})\text{, }\qquad\forall\phi_{h}\in V_{h}. \label{fullydiscrete2}%
\end{align}

\begin{theorem}
Assume $v_{h}^{n} = w_{h}^{n}$. Then for any positive $H$ and $\chi$, we have
\begin{gather}
||| v_{h}^{n+1}-u(t_{n+1})|||^{2} \leq||| w_{h}^{n+1}-u(t_{n+1})|||^{2}
+C\Delta t\chi\breve{C}^{2} (C_{P}^{2} +\Delta t\nu)\left(  1+\frac{1}%
{\beta^{LBB}}\right)  ^{2} h^{2k} \Vert u\Vert_{\infty, k+1}^{2}.\nonumber
\end{gather}

\end{theorem}

\begin{proof}
The true solutions of the NSE $u$ satisfy
\begin{gather}
\left(\frac{u^{n+1}-u^{n}}{\Delta t},\phi_{h}\right)+b^{\ast}(u^{n+1},u^{n+1},\phi
_{h})+\nu(\nabla u^{n+1},\nabla\phi_{h})-(p^{n+1},\nabla\cdot\phi
_{h})\label{eq:convtrue}\\
=(f^{n+1},\phi_{h})+Intp(u^{n+1};\phi_{h})\text{ , }\qquad\text{for all }%
\phi_{h}\in V_{h}\text{ ,}\nonumber
\end{gather}
where $Intp(u^{n+1};\phi_{h})$ is a consistency error
\[
Intp(u^{n+1};\phi_{h})
=\left(\frac{u^{n+1}-u^{n}}{\Delta t}-\partial_{t}%
u(t^{n+1}),\phi_{h}\right)\text{ .}%
\]

\noindent Subtracting \eqref{eq:convtrue} from \eqref{fullydiscrete0} and
\eqref{fullydiscrete2} respectively yields two error equations for%

\[
e^{n+1}=w_{h}^{n+1}-u(t_{n+1}) \qquad\text{and} \qquad\varepsilon^{n+1}%
=v_{h}^{n+1}-u(t_{n+1}).
\]

\begin{align}
&  \left(\frac{\varepsilon^{n+1}-\varepsilon^{n}}{\Delta t},\phi_{h}\right)+\nu
(\nabla\varepsilon^{n+1},\nabla\phi_{h})-b^{\ast}(u^{n+1},u^{n+1},\phi
_{h})+b^{\ast}(v^{n}_{h},v_{h}^{n},\phi_{h})\nonumber\\
&  +(p^{n+1},\nabla\cdot\phi_{h})+\chi(I_{H}\varepsilon^{n+1},I_{H}\phi_{h})
=-Intp(u^{n+1};\phi_{h})\text{ .} \label{eq:err}%
\end{align}

\begin{align}
&  \left(\frac{e^{n+1}-e^{n}}{\Delta t},\phi_{h}\right)+\nu(\nabla e^{n+1},\nabla\phi
_{h})-b^{\ast}(u^{n+1},u^{n+1},\phi_{h})+b^{\ast}(w^{n}_{h},w_{h}^{n},\phi
_{h})\nonumber\\
&  +(p^{n+1},\nabla\cdot\phi_{h}) =-Intp(u^{n+1};\phi_{h})\text{ .}
\label{eq:err0}%
\end{align}

\noindent Equating the left hand sides the above two equations and assuming
$v^{n}_{h}=w^{n}_{h}$ and thus $\varepsilon^{n}=e^{n}$, we have%

\begin{align}
&  \left(\frac{\varepsilon^{n+1}}{\Delta t},\phi_{h}\right)+\nu(\nabla\varepsilon
^{n+1},\nabla\phi_{h}) +\chi(I_{H}\varepsilon^{n+1},I_{H}\phi_{h})
=\left(\frac{e^{n+1}}{\Delta t},\phi_{h}\right)+\nu(\nabla e^{n+1},\nabla\phi_{h}) \text{
.} \label{eq:err1}%
\end{align}

\noindent Let $P^{h} u^{n}\in V_{h}$ be the modified H1 projection of $u^{n}%
$onto $V_{h}$ and denote%
\begin{gather*}
e^{n}=w_{h}^{n}-u^{n}=\left(  P^{h} u^{n}-u^{n}\right)  +\left(  w_{h}%
^{n}-P^{h} u^{n}\right)  =\eta^{n}+\zeta_{h}^{n}\text{ ,}\\
\epsilon^{n}=v_{h}^{n}-u^{n}=\left(  P^{h} u^{n}-u^{n}\right)  +\left(
v_{h}^{n}-P^{h} u^{n}\right)  =\eta^{n}+\xi_{h}^{n}\text{ .}%
\end{gather*}
\noindent We then have $\zeta_{h}^{n}\in V_{h}, \xi_{h}^{n}\in V_{h}$.
\eqref{eq:err1} becomes%

\begin{align}
&  \left(\frac{\xi_{h}^{n+1}}{\Delta t},\phi_{h}\right)+\nu(\nabla\xi_{h}^{n+1}%
,\nabla\phi_{h}) +\chi(I_{H}\xi_{h}^{n+1},I_{H}\phi_{h}) +\chi(I_{H}\eta
^{n+1},I_{H}\phi_{h})\nonumber\\
&  \qquad
=\left(\frac{\zeta_{h}^{n+1}}{\Delta t},\phi_{h}\right)+\nu(\nabla\zeta
_{h}^{n+1},\nabla\phi_{h}) \text{ .} \label{eq:err2}%
\end{align}

\noindent Setting $\phi_{h}=\xi_{h}^{n+1}$ in $V_{h}$ yields%

\begin{gather}
\|\xi_{h}^{n+1}\|^{2} +\Delta t \nu\|\nabla\xi_{h}^{n+1}\|^{2} +\Delta
t\chi(I_{H}\xi_{h}^{n+1},I_{H}\xi_{h}^{n+1}) = (\zeta_{h}^{n+1},\xi_{h}^{n+1})
+\Delta t\nu(\nabla\zeta^{n+1},\nabla\xi^{n+1})\\
-\Delta t\chi(I_{H}\eta^{n+1},I_{H}\xi_{h}^{n+1}) =\frac{1}{2} \Vert\zeta
_{h}^{n+1}\Vert^{2} + \frac{1}{2}\Vert\xi_{h}^{n+1}\Vert^{2} -\frac{1}{2}
\Vert\zeta_{h}^{n+1}-\xi_{h}^{n+1}\Vert^{2}\nonumber\\
+\Delta t\nu\left(  \frac{1}{2} \Vert\nabla\zeta_{h}^{n+1}\Vert^{2} + \frac
{1}{2}\Vert\nabla\xi_{h}^{n+1}\Vert^{2} -\frac{1}{2} \Vert\nabla\zeta
_{h}^{n+1}-\nabla\xi_{h}^{n+1}\Vert^{2}\right)  -\Delta t\chi(I_{H}\eta
^{n+1},I_{H}\xi_{h}^{n+1}).\nonumber
\end{gather}

\noindent Using the fact that $\zeta_{h}^{n+1}-\xi_{h}^{n+1}=e^{n+1}%
-\varepsilon^{n+1}=w_{h}^{n+1}-v_{h}^{n+1}$, the above equation can be
rewritten as%

\begin{gather}
||| \xi_{h}^{n+1}|||^{2} +2\Delta t\chi\Vert I_{H}\xi_{h}^{n+1}\Vert^{2} +|||
w_{h}^{n+1}-v_{h}^{n+1} |||^{2} =||| \zeta_{h}^{n+1}|||^{2} -2\Delta
t\chi(I_{H}\eta^{n+1},I_{H}\xi_{h}^{n+1}).
\end{gather}

\noindent We apply Cauchy-Schwarz and Young's inequalities and \eqref{bound1}
to bound the last term in the above equation.
\begin{align}
-\chi\left(  I_{H}\eta^{n+1},I_{H}\xi_{h}^{n+1}\right)   &  \leq\chi\left\Vert
I_{H}\eta^{n+1}\right\Vert \left\Vert I_{H}\xi_{h}^{n+1}\right\Vert
\nonumber\\
&  \leq\dfrac{1}{2}\chi\left\Vert I_{H}\eta^{n+1}\right\Vert ^{2}+ \frac{1}%
{2}\chi\left\Vert I_{H}\xi_{h}^{n+1}\right\Vert ^{2}\\
&  \leq\dfrac{1}{2}\chi\breve{C}^{2}\left\Vert \eta^{n+1}\right\Vert
^{2}+\frac{1}{2}\chi\left\Vert I_{H}\xi_{h}^{n+1}\right\Vert ^{2}.\nonumber
\end{align}

\noindent Then we have
\begin{gather}
|||\xi_{h}^{n+1}|||^{2} +\Delta t\chi\Vert I_{H}\xi_{h}^{n+1}\Vert^{2} +|||
w_{h}^{n+1}-v_{h}^{n+1} |||^{2} \leq|||\zeta_{h}^{n+1}|||^{2} + \Delta
t\chi\breve{C}^{2}\left\Vert \eta^{n+1}\right\Vert ^{2},
\end{gather}

\noindent which can be further reduced to%

\begin{gather}
|||\xi_{h}^{n+1}|||^{2} \leq|||\zeta_{h}^{n+1}|||^{2} + \Delta t\chi\breve
{C}^{2}\left\Vert \eta^{n+1}\right\Vert ^{2}.
\end{gather}

\noindent Recall that $e^{n}=\eta^{n}+\zeta_{h}^{n}, \varepsilon^{n}=\eta
^{n}+\xi_{h}^{n}$. We have%

\begin{gather}
|||\varepsilon^{n+1}|||^{2} =|||\xi_{h}^{n+1}|||^{2} + |||\eta^{n+1} |||^{2}\\
\leq|||\zeta_{h}^{n+1}|||^{2} + \Delta t\chi\breve{C}^{2}\left\Vert \eta
^{n+1}\right\Vert ^{2} + |||\eta^{n+1} |||^{2}\nonumber\\
\leq||| e^{n+1}|||^{2} + \Delta t\chi\breve{C}^{2}\left\Vert \eta
^{n+1}\right\Vert ^{2}\nonumber\\
\leq||| e^{n+1}|||^{2} + C\Delta t\chi\breve{C}^{2} (C_{P}^{2} +\Delta
t\nu)\left(  1+\frac{1}{\beta^{LBB}}\right)  ^{2} h^{2k} \Vert u\Vert_{\infty,
k+1}^{2}.\nonumber
\end{gather}

\end{proof}

\section{Experiments}

\label{sec5}

In this section, three numerical tests are conducted to investigate the effect
of nudging with sparse data in space and infrequent (at only 1 timestep) in
time. The basic experiments involve comparing three simulations:

\begin{itemize}
\item Simulation 1: A finer mesh simulation with a higher accuracy time
discretization to provide a reference solution. The initial condition for this
simulation is considered to be correct.

\item Simulation 2: A coarser mesh simulation with perturbed initial
condition. The time discretization uses a fully explicit treatment of the
nonlinearity in accord with the theory. This simulation is the non-nudged result.

\item Simulation 3: At each time step of simulation 2, a nudging term is added
and the time step is repeated. This means simulation 3 only investigates the
(non-accumulated) effect of data added each step to the result of the no-data
simulation, also in accord with the theory.
\end{itemize}

These 3 basic simulations lead to more natural questions beyond the theory
herein. The first is whether the improvement persists with linearly implicit
treatment of the nonlinear term of the NSE. We address this question in
Experiments 5.2 and 5.3. The second is how much better the CDA errors would be
if the error improvement is allowed to accumulate at each step. We address
this question in Experiment 5.3.

The first test computes the above second and third simulations on a unit
square with known exact solution replacing simulation 1. The nonlinear term is
fully explicit. The exact solution is used to provide the reference data for
the nudging term. The second test computes solutions for the pipe cavity
problem at low and high Reynolds numbers. The nonlinear term is fully explicit
for low Reynolds number and semi-implicit for high Reynolds number. The third
test simulates rotational flow in a domain with an offset cylinder at a higher
Reynolds number. The nonlinear term for the third test is semi-implicit. Tests
2 and 3 have no known exact solutions, so simulation 1 above was used to
provide the reference data.

In tests (not included herein) with sparse data at only one timestep and small $\chi$, the improvement seen was small, as expected. To produce figures with a visual difference, the tests increased $\chi$ but kept the data sparse and only at one timestep.

\subsection{Analytical solution on unit square}

The first test is adapted from\ \cite{ref5}. We consider an analytical
solution on unit square domain $\Omega=(0,1)^{2}$. The analytical velocity and
pressure are
\[
u(x,y,t)=e^{t}(\cos y,\sin x)^{T}\text{ and }p(x,y,t)=(x-y)(1+t).
\]
We compute the forcing term $f(x,t)$ using the analytical solutions.
Simulations 2 and 3 are computed on a fine mesh with 43,266 degrees of freedom
(dof). We assume $I_{H}$ is $L^{2}$ projection onto a continuous piecewise
quadratic velocity space. The observation spacing $H$ is $H=0.0441942$. We set
$\Delta t=1/16$ and $\chi=10^{4}$. We take Scott-Vogelius $(P2-P1dc)$ finite
element pair with a barycenter refined mesh, e.g. \cite{GS19}. The initial
conditions for $u$ is $u(x,y,0)$ and for $w$ and $v$ are $w^{0}=v^{0}%
=u(x,y,0)+(10^{-3},10^{-3})^{T}$. The nonlinear term for both simulations is
fully explicit.

The simulation results are recorded for relative errors (since the solution
grows exponentially) in the weighted norm $|||\phi|||^{2}$. The errors are
computed relative to the exact solution.

\begin{figure}[ptbh]
\centering
\includegraphics[width=0.75\linewidth]{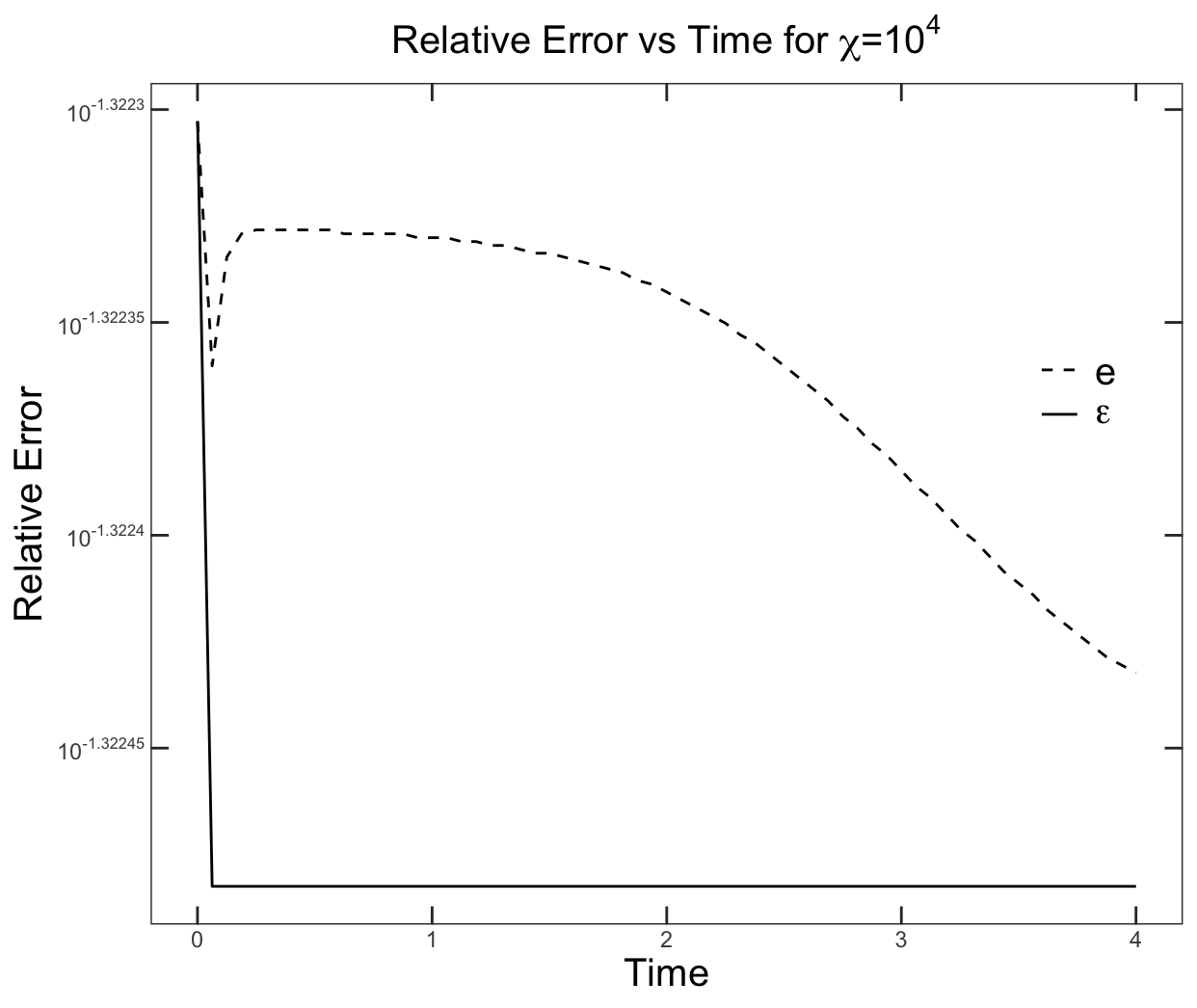}\caption{Nudging decreases
error. Decrease is small for stable problems.}%
\label{fig:Test1}%
\end{figure}In this test, the exact solution is smooth and (seems to be)
stable. Thus, the test's main points are to verify correctness of
implementation and that the added nudging term does increase accuracy. At each
time step, the nudging correction with sparse and infrequent data does
increase accuracy but, for this stable problem, the increase is minor.

\FloatBarrier

\subsection{Pipe cavity flow problem}

The second test is adapted from Ervin, Layton, and Maubach \cite{L00} and, due
to the boundary conditions, is beyond the theorem proven herein. The domain is
a horizontal pipe of length 8 and height 1 with a rectangular cavity on top of
the pipe of length 6 and height 5. The domain is defined by
\[
\Omega=([0,8]\times\lbrack0,1])\cup([1,7]\times\lbrack1,6]).
\]
We impose no-slip boundary conditions on the walls and no force $f(x,y,t)=0$.
We have inflow of fluid, $u(x,y)=\min(t,1)(4y(1-y))^{T}$, from the left end of
the pipe and outflow at the right end. We apply the standard do-nothing
condition, \cite{HRT96}, at the outflow of the pipe at $x=8$. Others, e.g.
\cite{ABL16}, are certainly possible.

The coarse solutions are computed on a mesh with 10,433 dof. We set the final
time $T=10$, time step size $\Delta t=1/500,\nu=10^{-2},L=1,U=1$, and
$Re=\frac{LU}{\nu}$. The initial velocity conditions are $w^{0}=v^{0}%
=(2\times10^{-3}, 2\times10^{-3})^{T}$. This initial condition is incompatible
with the inflow boundary condition so there is a large error for a few initial
timesteps. \ We set $I_{H}$ to be $L^{2}$ projection and set $\chi=10^{4}$.
The observation spacing $H$ is $H=0.31241$. We take Taylor-Hood $(P2-P1)$
finite element pair for the velocity and pressure spaces.

The reference solution $u$ is approximated using Direct Numerical Simulation
(DNS), computed on a finer mesh with 41,194 dof and with the second order BDF2
time discretization. Experiment parameters for the DNS solution are kept the
same as parameters for the coarse solutions. The initial velocity condition is
$u(x,y,0)=(10^{-3},10^{-3})^{T}$. The nonlinear term for the reference
simulation is fully explicit with AB2: $(2u^{n}-u^{n-1})\cdot\nabla
(2u^{n}-u^{n-1})$.

The simulation results are recorded for errors in the weighted norm
$|||\phi|||^{2}$ and computed relative to the reference solutions.

\begin{figure}[ptbh]
\centering
\includegraphics[width=0.75\linewidth]{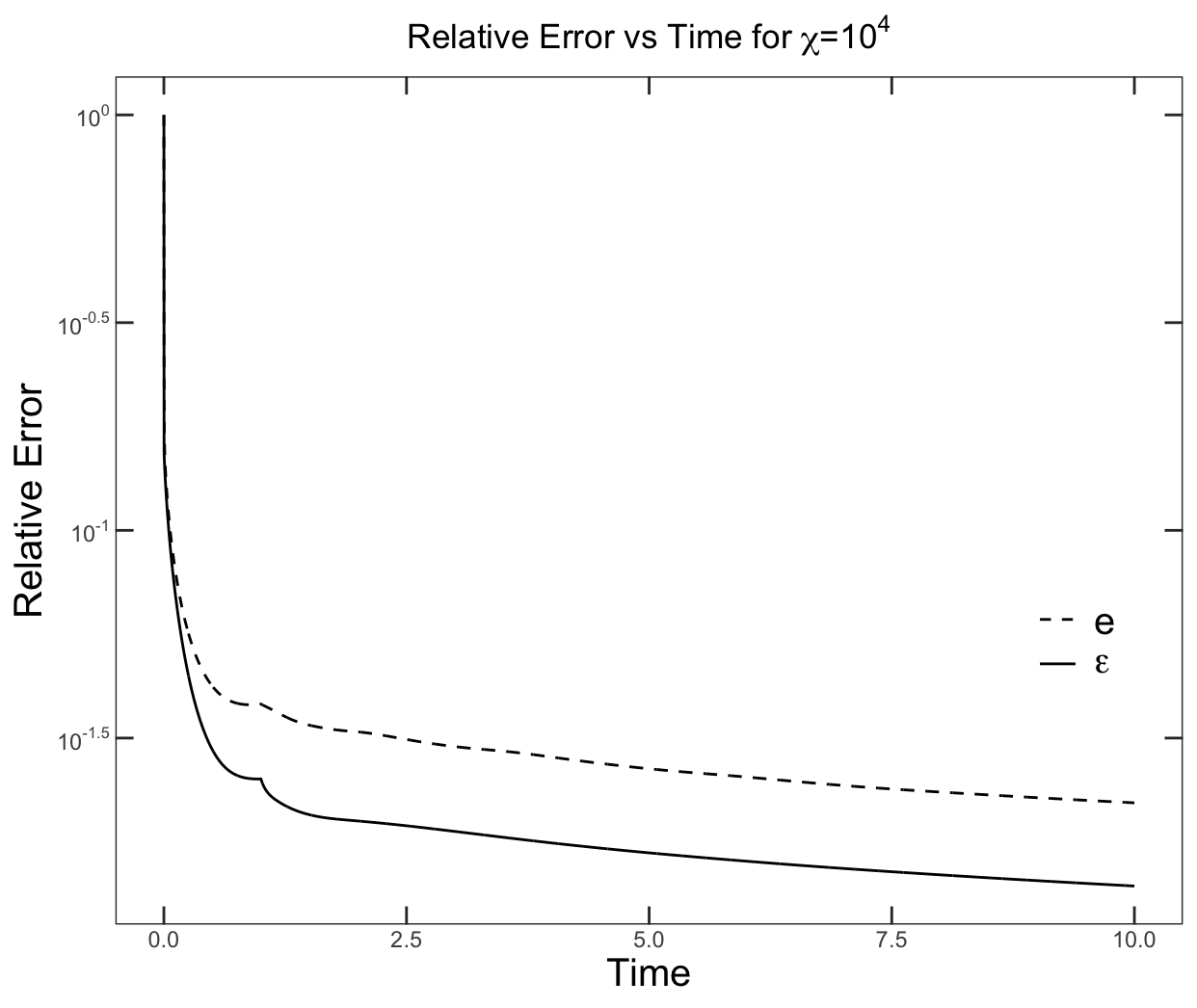}\caption{At each time step,
the error curve with nudging is below the error curve without nudging.}%
\label{fig:Test2}%
\end{figure}

Figure 5.2 shows the errors with and without data begin at $\mathcal{O}(1)$
due to the incompatibility of the initial condition with the inflow data. Then
errors decrease over time. The error curve without nudging reaches
$\mathcal{O}(10^{-1.5})$ at $t=1$. The error curve with nudging reaches around
$\mathcal{O}(10^{-2})$ at $t=1$. The error curve with nudging is strictly
below the error curve without nudging. Adding sparse and infrequent data
decreases error.

\begin{figure}[htbp]
    \centering
    \begin{subfigure}[b]{0.45\textwidth}
        \centering
        \caption*{$t=0$} 
        \includegraphics[width=\textwidth]{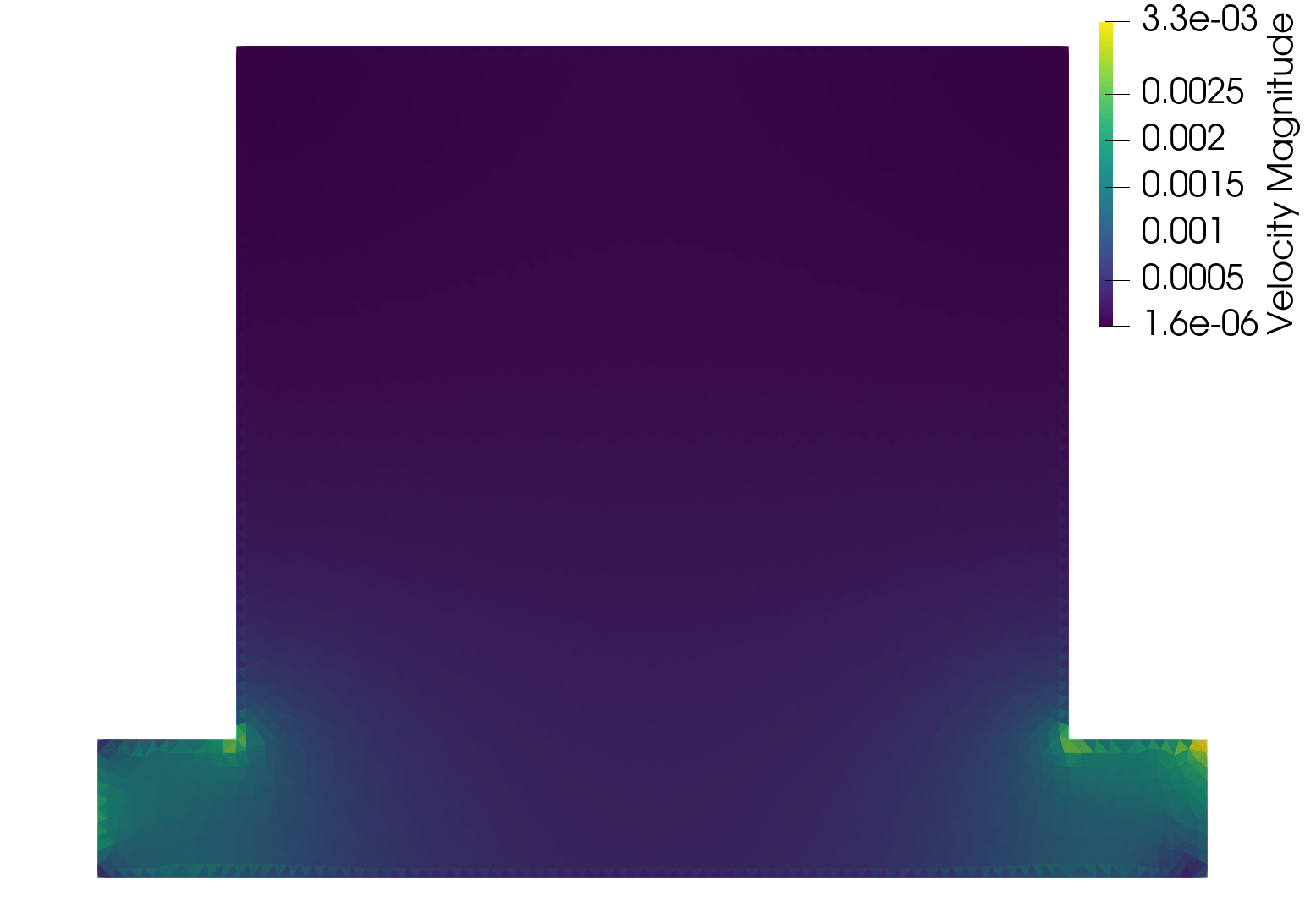}
    \end{subfigure}
    \hfill
    \begin{subfigure}[b]{0.45\textwidth}
        \centering
        \caption*{$t=3$}
        \includegraphics[width=\textwidth]{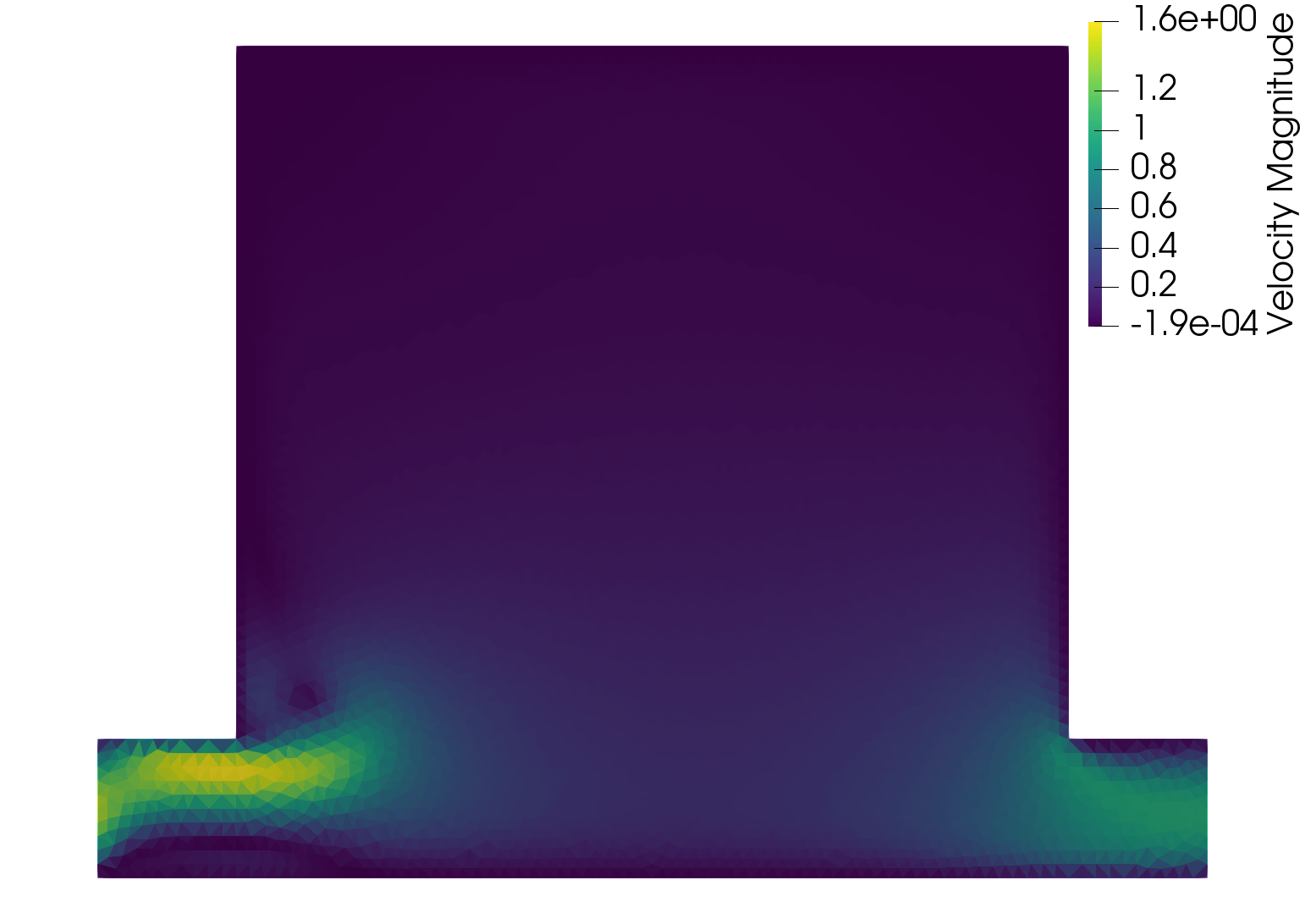}
    \end{subfigure}
    
    \vspace{0.5cm} 
    
    \begin{subfigure}[b]{0.45\textwidth}
        \centering
        \caption*{$t=6$}
        \includegraphics[width=\textwidth]{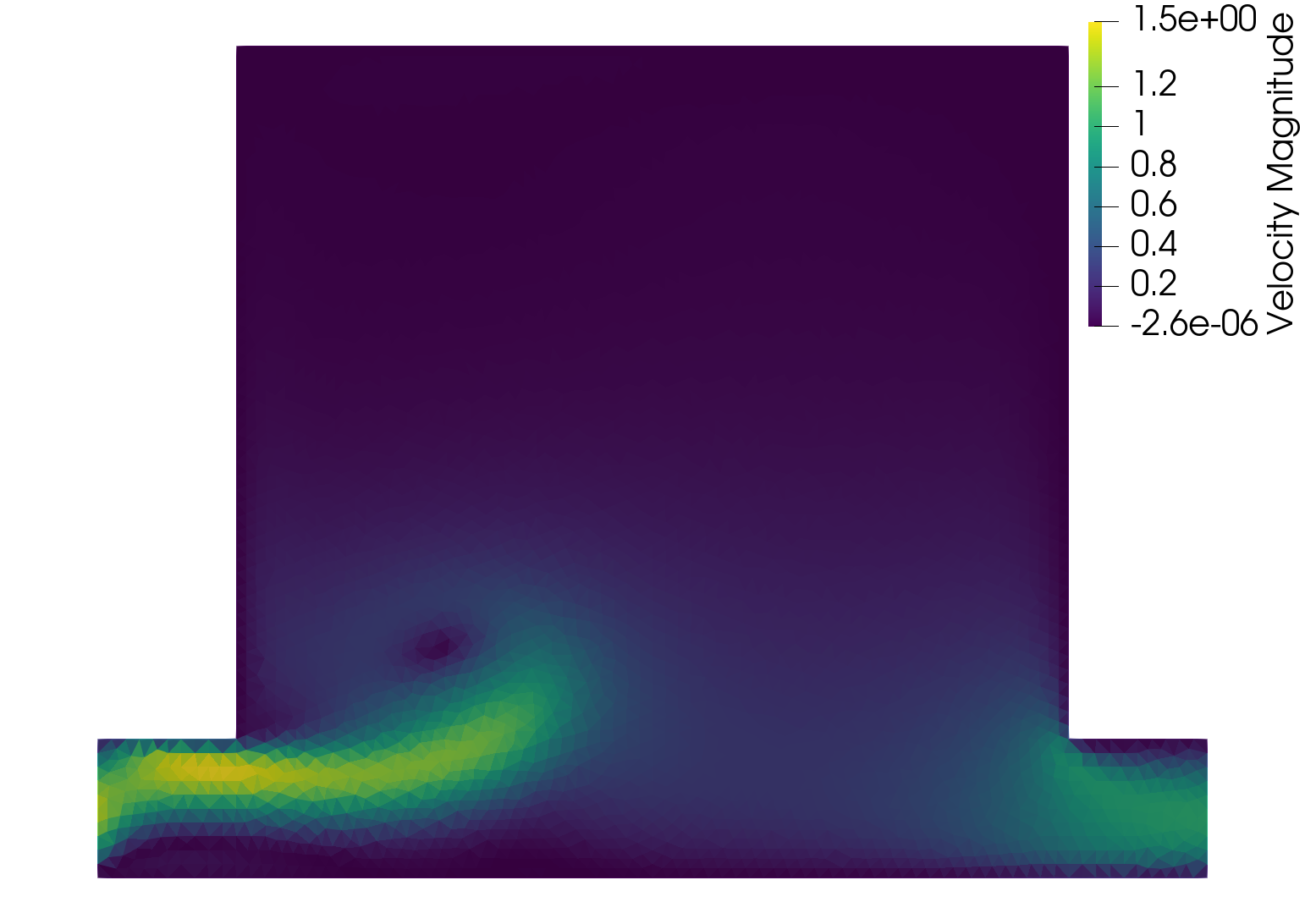}
    \end{subfigure}
    \hfill
    \begin{subfigure}[b]{0.45\textwidth}
        \centering
        \caption*{$t=9$}
        \includegraphics[width=\textwidth]{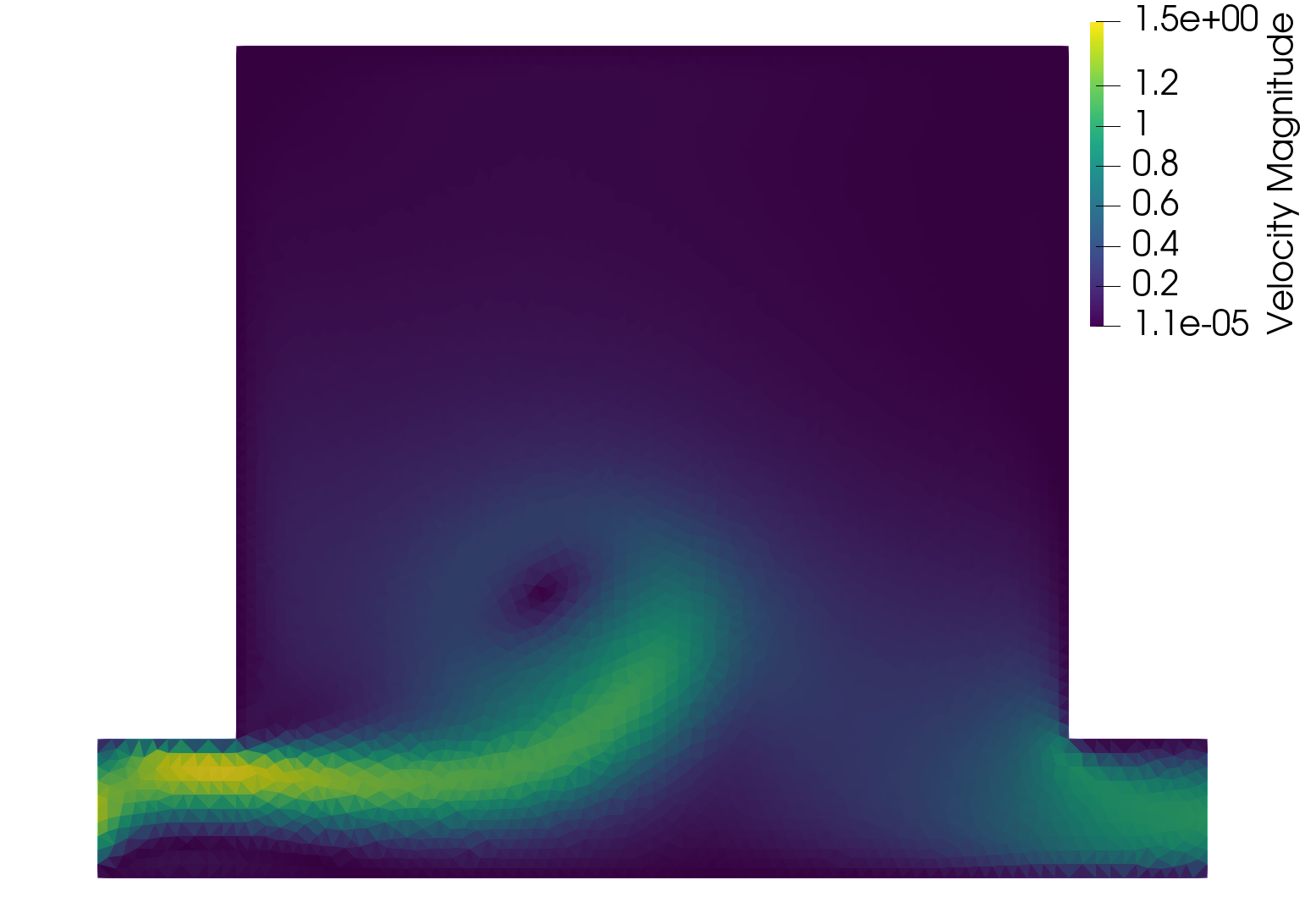}
    \end{subfigure}
    
    \vspace{0.3cm}
    \caption{Contour plots of the reference solution's velocity magnitudes at times 0, 3, 6, and 9.}
    \label{fig:four_images_test2}
\end{figure}

\FloatBarrier

\subsubsection{Higher Reynolds number flow}

In this subsection, we repeat the same test with $\mathcal{R}e=1000$. We set
$\Delta t=1/100$. In the first order, coarser mesh, simulations 2 and 3, the
nonlinear term is treated using the second order, linearly implicit,
skew-symmetric formulation (\ref{eq:bstar})
\[
((2v^{n}-v^{n-1})\cdot\nabla v^{n+1},z)+\frac{1}{2}((\nabla\cdot
(2v^{n}-v^{n-1}))v^{n+1},z),\text{ $\forall z\in X_{h}$}.
\]
The nonlinear term for the reference solution is treated the same as above.
All other experiment parameters and initial conditions remain the same.

\begin{figure}[ptbh]
\centering
\includegraphics[width=0.75\linewidth]{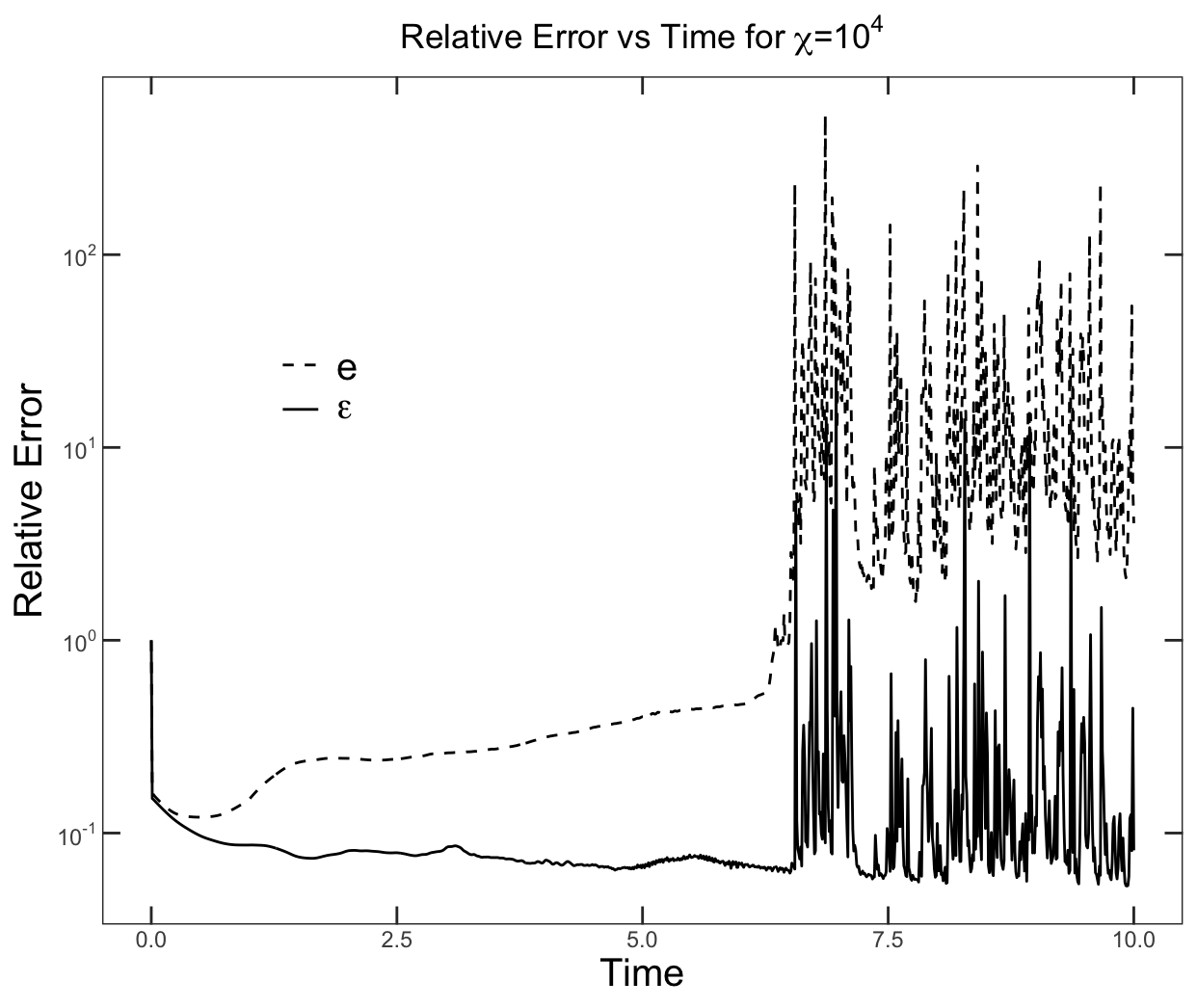}\caption{The error curve
with nudging is below the error curve without nudging except at $t=8.94$.}%
\label{fig:Test2_1}%
\end{figure}

Figure 5.4 shows that the error without nudging starts at $\mathcal{O}(1)$ and
decreases after initial calculation. Then it slowly increases and experiences
exponential growth at around $t=7$ and high-frequency oscillation afterwards.
The error with nudging begins at $\mathcal{O}(1)$ and slowly decreases to
$\mathcal{O}(10^{-1})$. At $t=7$, the error grows exponentially and begins to
oscillate. The high-frequency oscillation for the error suggests a competition
between error growth in some modes and nudging-caused error decay in others.
Except at one timestep ($t=8.94$), the error curve with nudging is below the
error curve without nudging.

\begin{figure}[htbp]
    \centering
    \begin{subfigure}[b]{0.45\textwidth}
        \centering
        \caption*{$t=0$} 
        \includegraphics[width=\textwidth]{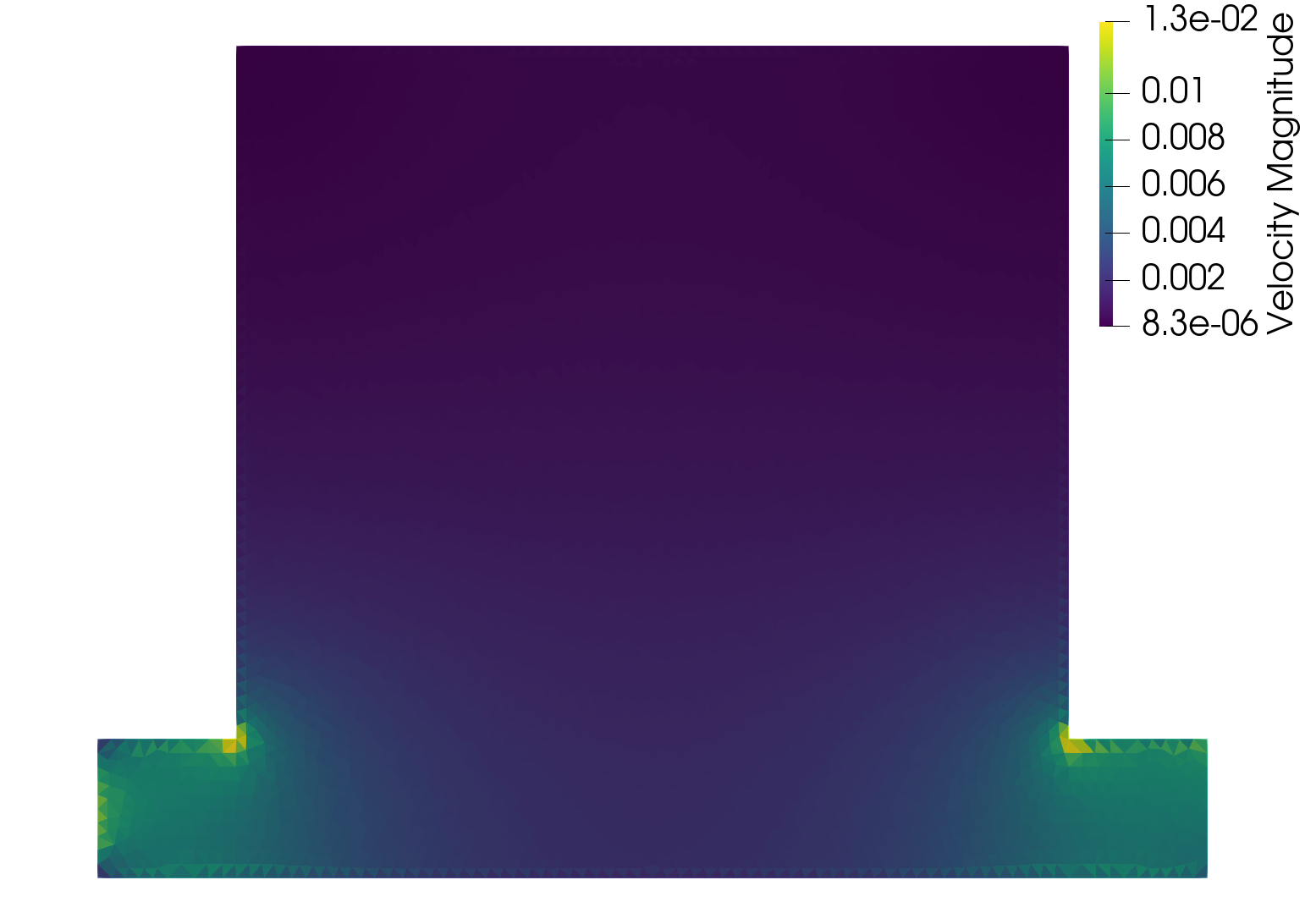}
    \end{subfigure}
    \hfill
    \begin{subfigure}[b]{0.45\textwidth}
        \centering
        \caption*{$t=3$}
        \includegraphics[width=\textwidth]{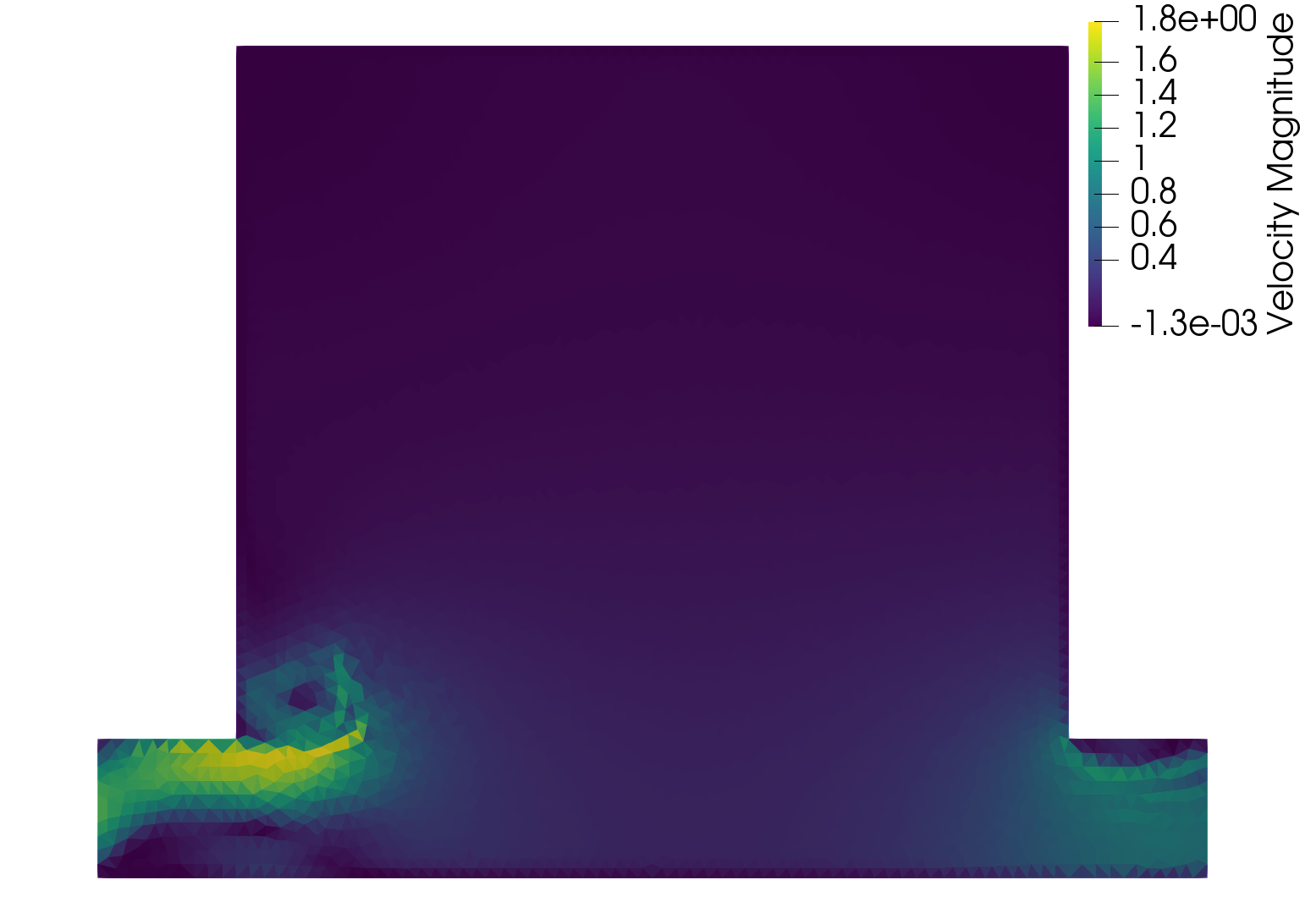}
    \end{subfigure}
    
    \vspace{0.5cm} 
    
    \begin{subfigure}[b]{0.45\textwidth}
        \centering
        \caption*{$t=6$}
        \includegraphics[width=\textwidth]{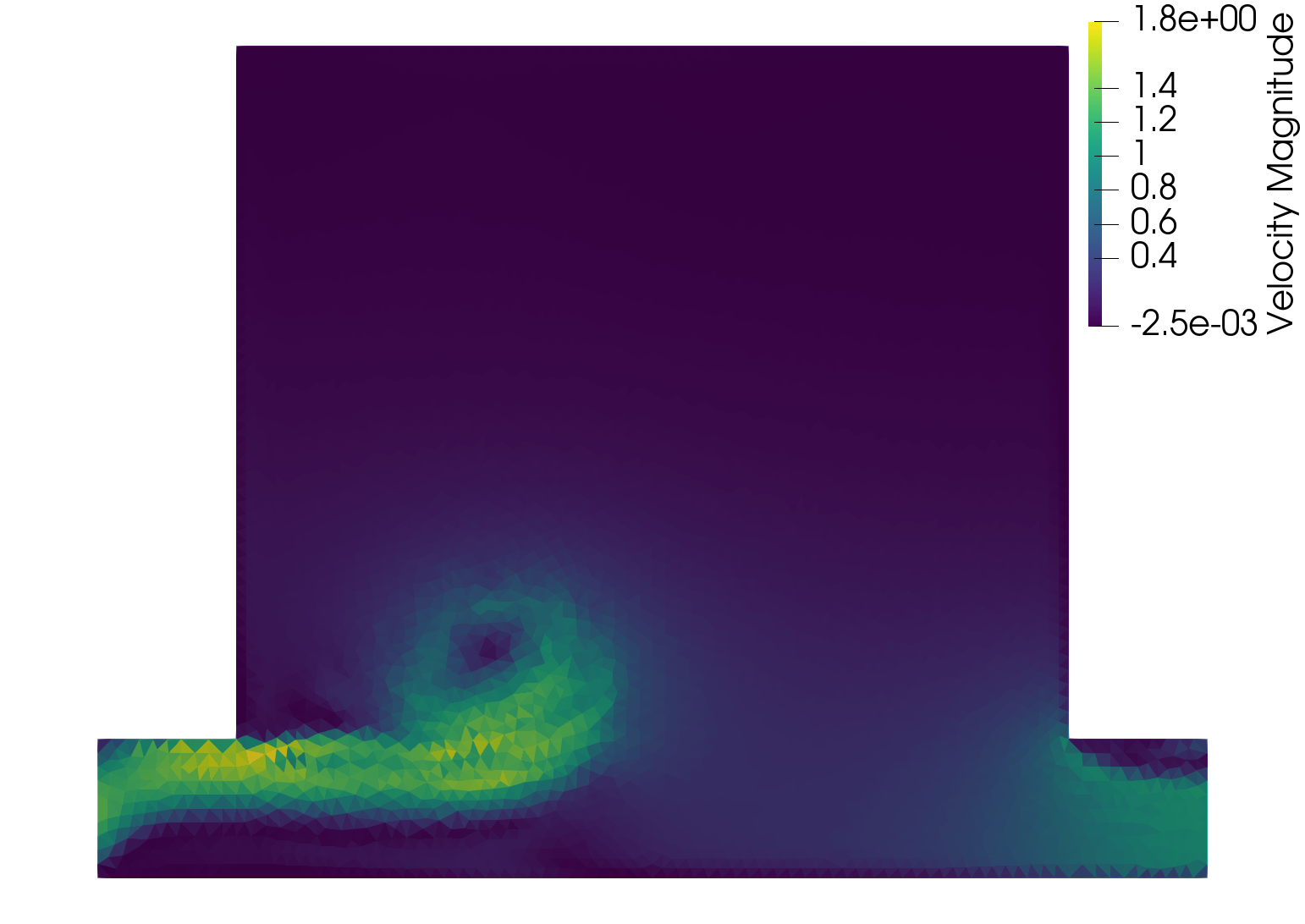}
    \end{subfigure}
    \hfill
    \begin{subfigure}[b]{0.45\textwidth}
        \centering
        \caption*{$t=9$}
        \includegraphics[width=\textwidth]{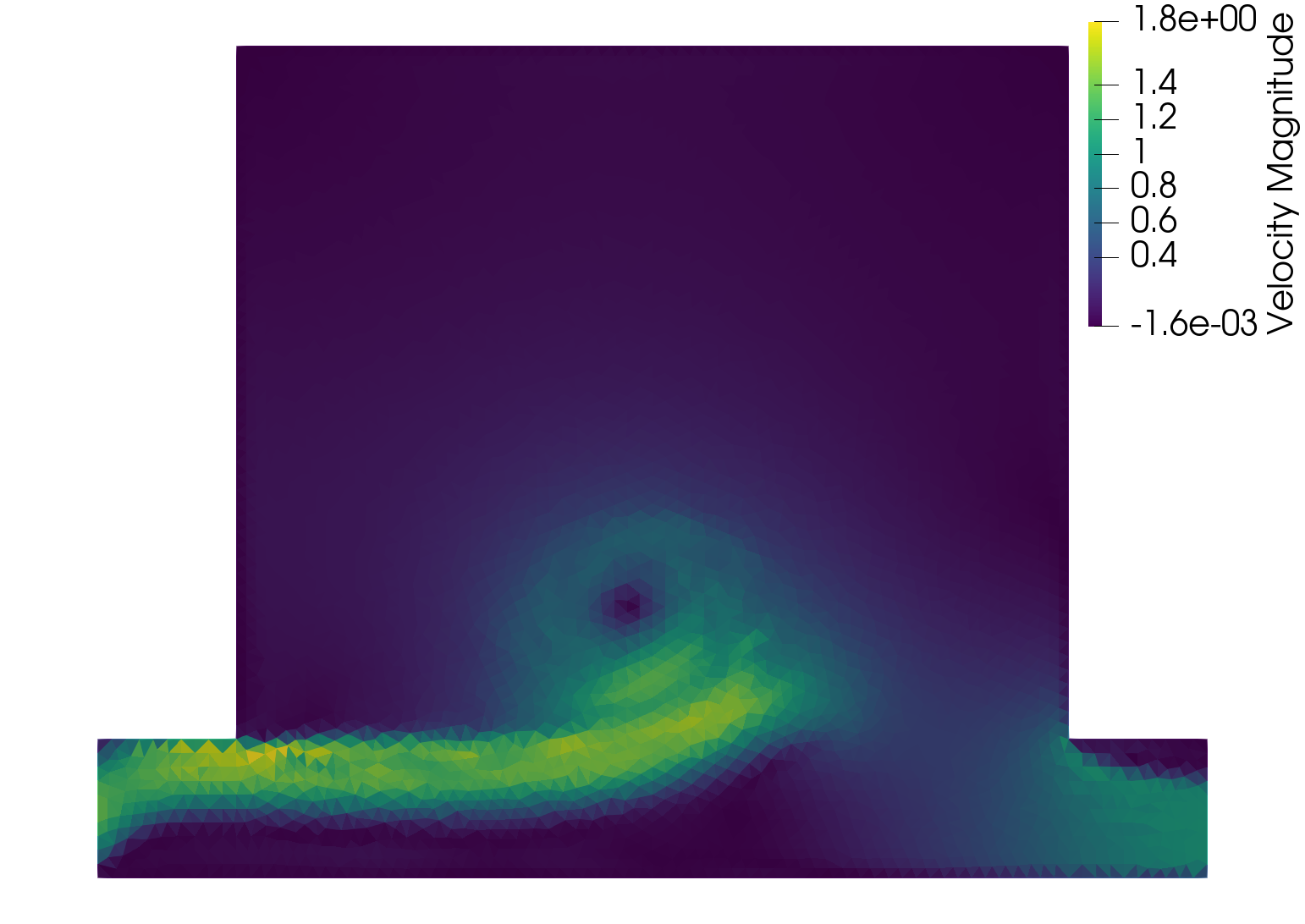}
    \end{subfigure}
    
    \vspace{0.3cm}
    \caption{Contour plots of the reference solution's velocity magnitudes at times 0, 3, 6, and 9.}
    \label{fig:four_images_test2.1}
\end{figure}

\FloatBarrier

\subsection{Flow between offset cylinders}

The third test is adapted from\ \cite{JL14} to data assimilation as in
\cite{ref4}. We consider a 2d flow at a higher Reynolds number without
analytical solution. The domain is a disk containing a smaller, off-centered
obstacle. The disk remains fixed in place. Let $r_{1}=1$ be the outer circle
radius and $r_{2}=0.1$ be the inner circle radius with center at $c=(1/2,0)$.
The domain is defined by
\[
\Omega=\{(x,y):x^{2}+y^{2}\leq r_{1}^{2}\text{ and }(x-c_{1})^{2}%
+(y-c_{2})^{2}\geq r_{2}^{2}\}.
\]
We assume no-slip boundary conditions, and the flow is driven by a rotational
force
\[
f(x,y,t)=(-4y\min(1,t)(1-x^{2}-y^{2}),4x\min(1,t)(1-x^{2}-y^{2}))^{T}.
\]
The rotational body force generates a rotational flow with a vortex street
past the inner disk. This vortex street re-interacts with the inner disk every
revolution. There is also a central , "polar" vortex that grows absorbing
smaller ones and quasi-periodically breaks up again.

The coarse solutions are approximated on a Delaunay generated mesh that has 75
mesh points on the outer boundary and 60 mesh points on the inner boundary. We
set final time $T=10$, time step size $\Delta t=1/100,\nu=10^{-3},L=1,U=1,$
and $Re=\frac{LU}{\nu}$. The initial conditions for $w$ and $v$ are
$w^{0}=v^{0}=(2\times10^{-3},2\times10^{-3})^{T}$. We impose the Dirichlet
boundary condition $w=v=0$ on $\partial\Omega$, the boundary of the inner and
outer disks. We take Taylor-Hood $(P2-P1)$ finite element pair for the
velocity and pressure spaces. The nonlinear convection term uses the
skew-symmetric form (\ref{eq:bstar}) implemented so as to be unconditionally
stable, second order and linearly implicit by
\[
((2v^{n}-v^{n-1})\cdot\nabla v^{n+1},z)+\frac{1}{2}((\nabla\cdot
(2v^{n}-v^{n-1}))v^{n+1},z),\text{ $\forall z\in X_{h}$}.
\]
We set $I_{H}$ to be the $L^{2}$ projection, satisfying the condition
$(I_{H}\phi,\phi)\geq0$, and set $\chi=10^{4}$. The observation spacing $H$ is
$H=0.114736$.

The reference solution $u$ is approximated using DNS, calculated on a finer
mesh with 120 mesh points on the outer boundary and 96 mesh points on the
inner boundary. The DNS solution uses the BDF2 time discretization scheme. The
nonlinear term uses the same skew-symmetric form (\ref{eq:bstar}) as above.
The DNS solution uses the same experiment parameters as used for the coarse
solutions. The initial condition for $u$ is $u(x,y,0)=(10^{-3},10^{-3})^{T}$.

The simulation results are recorded for errors in the weighted norm
$|||\phi|||^{2}$ and computed relative to the reference solution. We also
recorded a solution with accumulated nudging for comparison.

\begin{figure}[ptbh]
\centering
\includegraphics[width=0.75\linewidth]{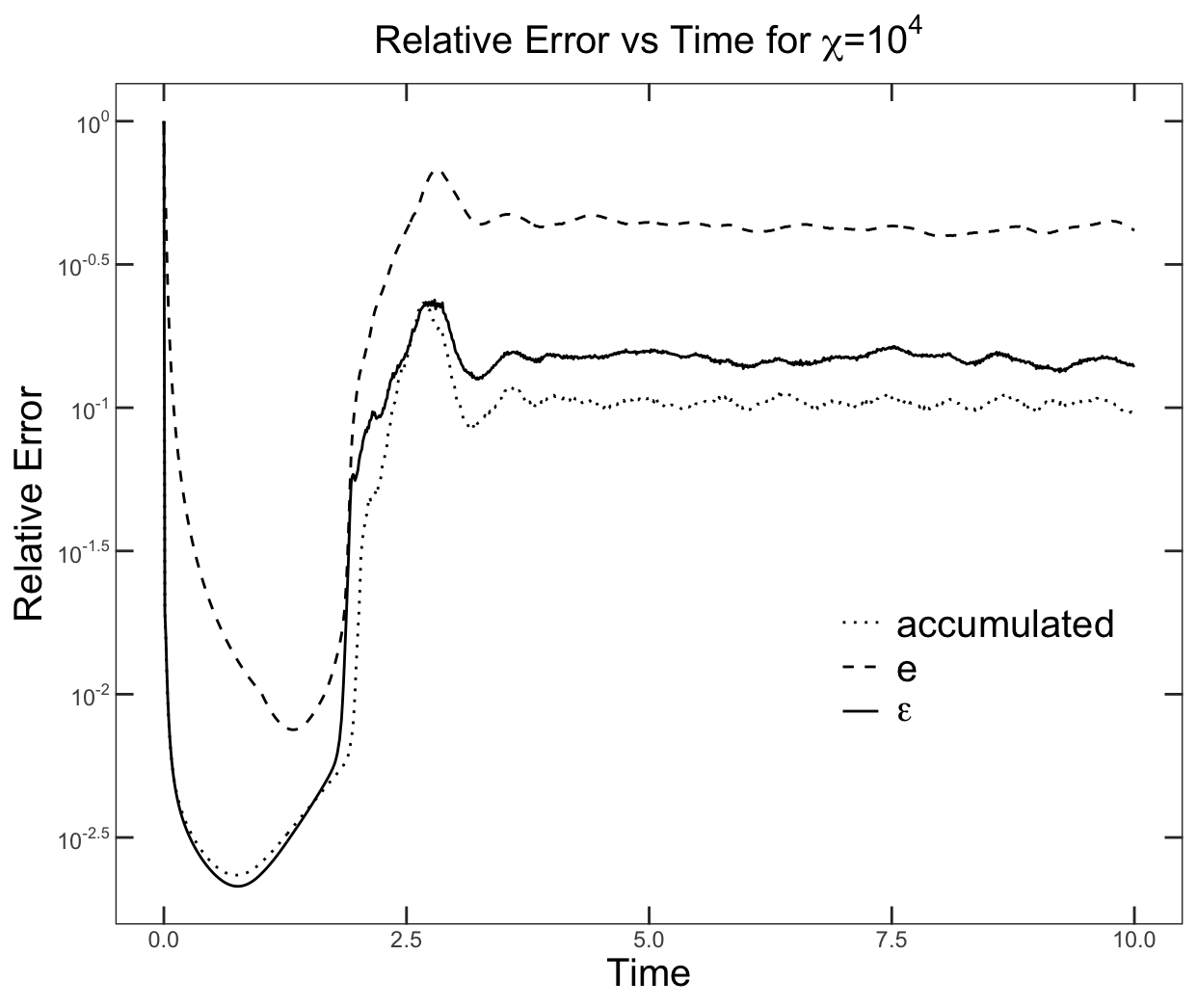}\caption{At each time step,
error with nudging is strictly less than error without nudging.}%
\label{fig:Test3}%
\end{figure}

In this test, the errors decrease quickly after the initial calculations to
$\mathcal{O}(10^{-2})$ for the error curve without nudging and to
$\mathcal{O}(10^{-2.5})$ for the error curve with nudging and accumulated
nudging. The error without nudging increases after $t=2$ and saturates at
$\mathcal{O}(10^{-0.5})$ at $t=3$. The error with nudging increases after
$t=1$ and saturates slightly above $\mathcal{O}(10^{-1})$ at $t=3$. The error
with accumulated nudging is similar to the error with nudging and saturate at
$\mathcal{O}(10^{-1})$ at $t=3$. At each time step, the error curve with
nudging is strictly below the error curve without nudging. This shows that
nudging at each time step also increases the predictability horizon. This
figure shows the effectiveness of accumulated nudging. The error curves
resemble the widely-accepted logistic model.

\begin{figure}[htbp]
    \centering
    \begin{subfigure}[b]{0.45\textwidth}
        \centering
        \caption*{$t=0$} 
        \includegraphics[width=\textwidth]{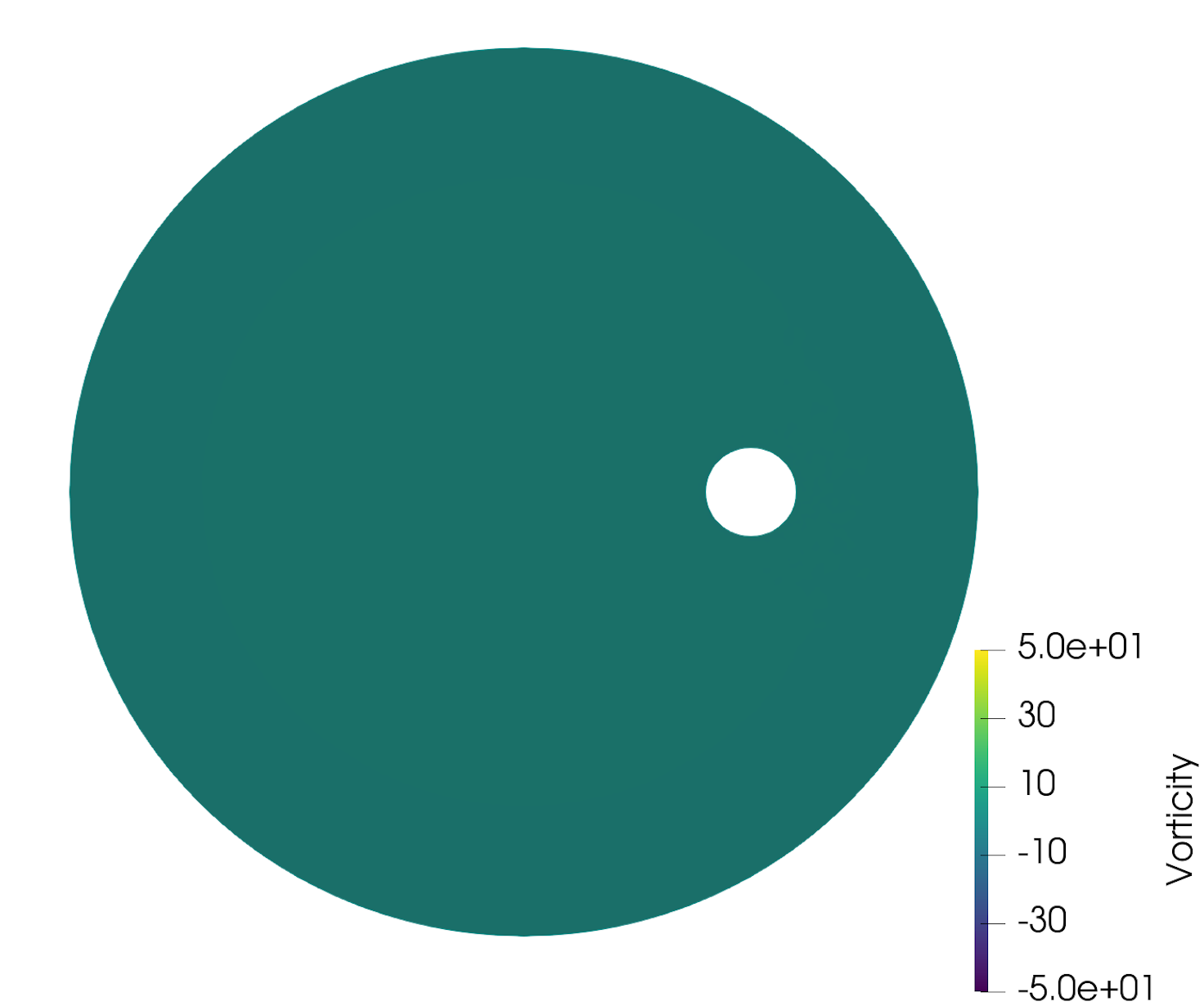}
    \end{subfigure}
    \hfill
    \begin{subfigure}[b]{0.45\textwidth}
        \centering
        \caption*{$t=2$}
        \includegraphics[width=\textwidth]{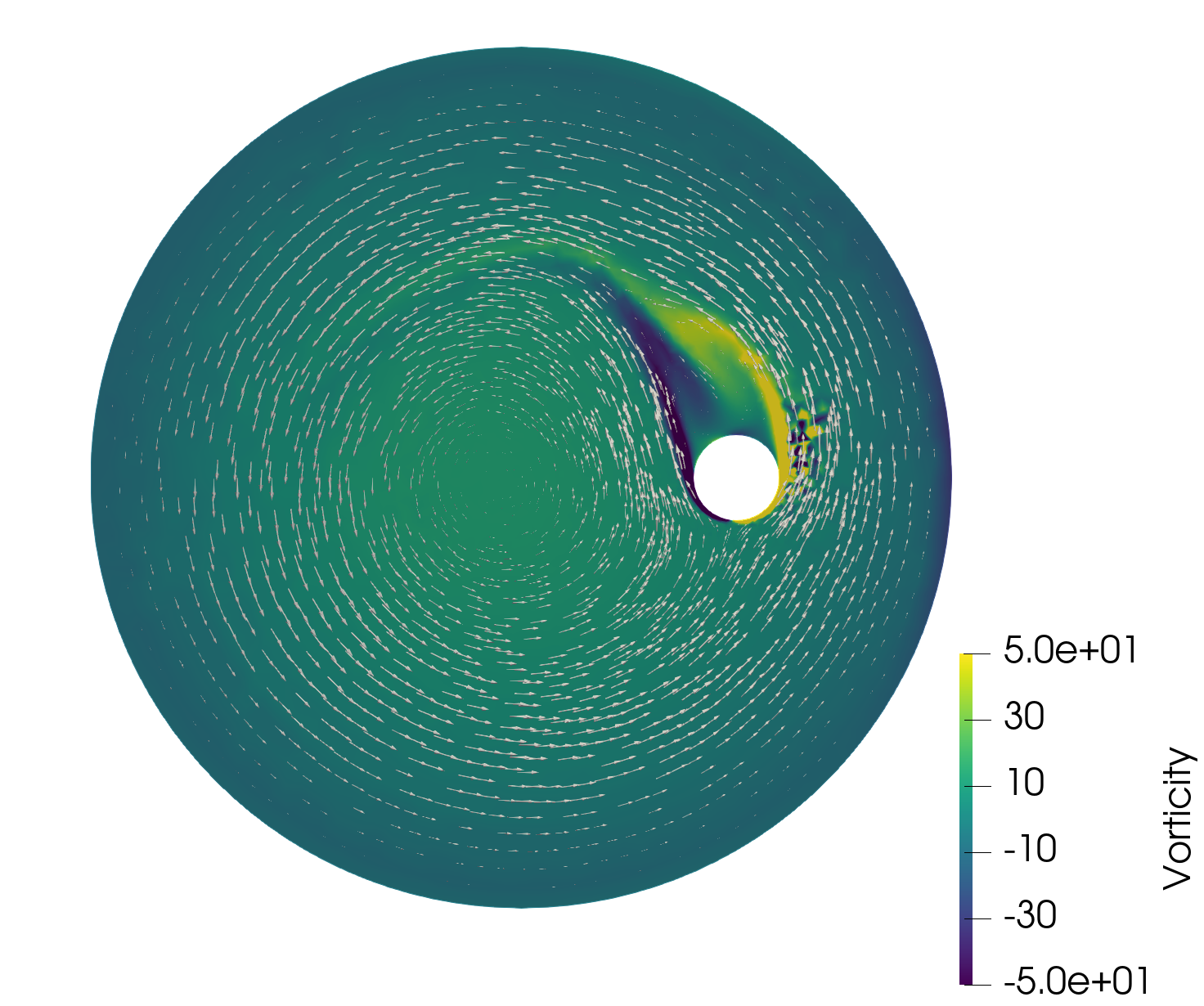}
    \end{subfigure}
    
    \vspace{0.5cm} 
    
    \begin{subfigure}[b]{0.45\textwidth}
        \centering
        \caption*{$t=4$}
        \includegraphics[width=\textwidth]{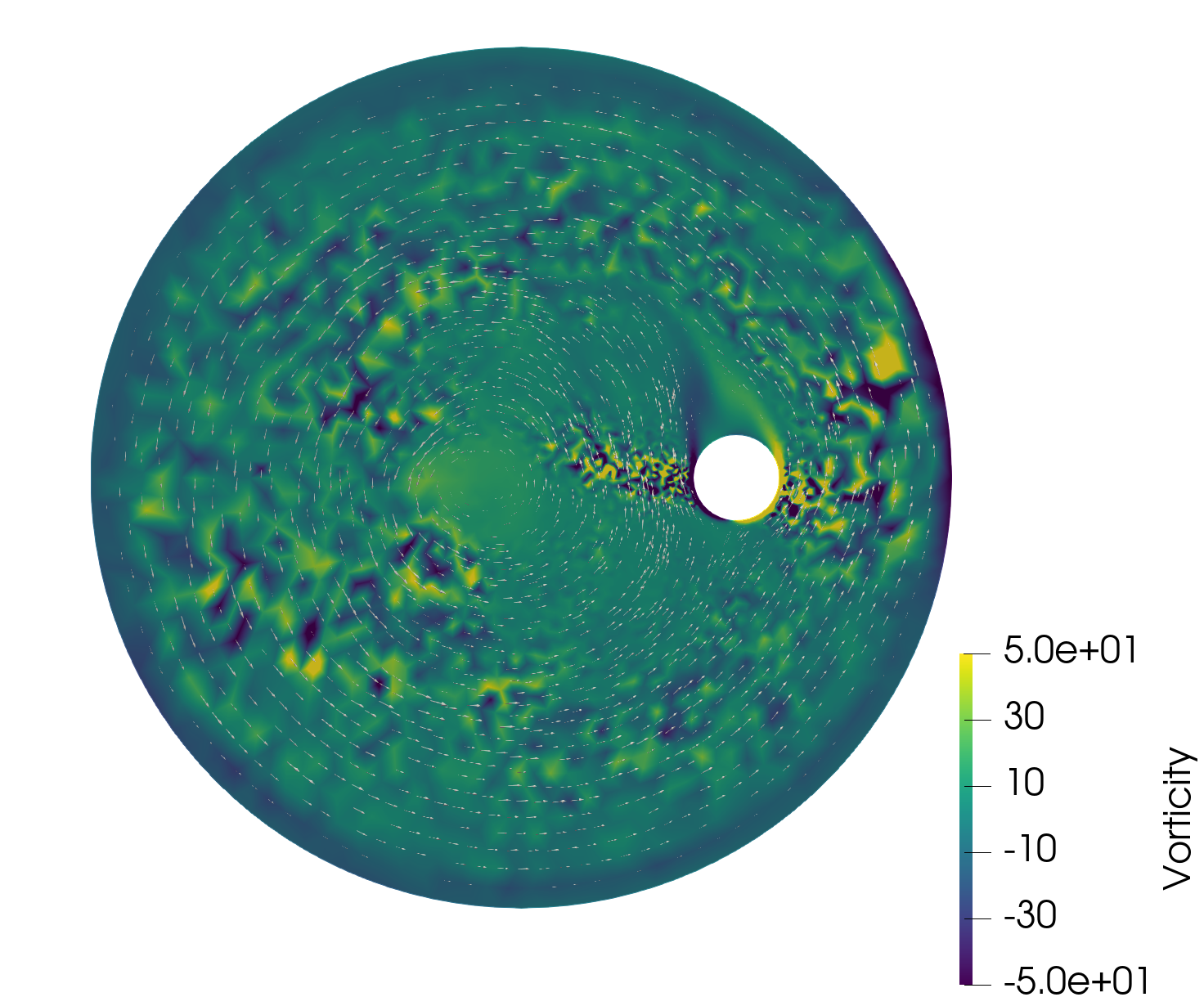}
    \end{subfigure}
    \hfill
    \begin{subfigure}[b]{0.45\textwidth}
        \centering
        \caption*{$t=6$}
        \includegraphics[width=\textwidth]{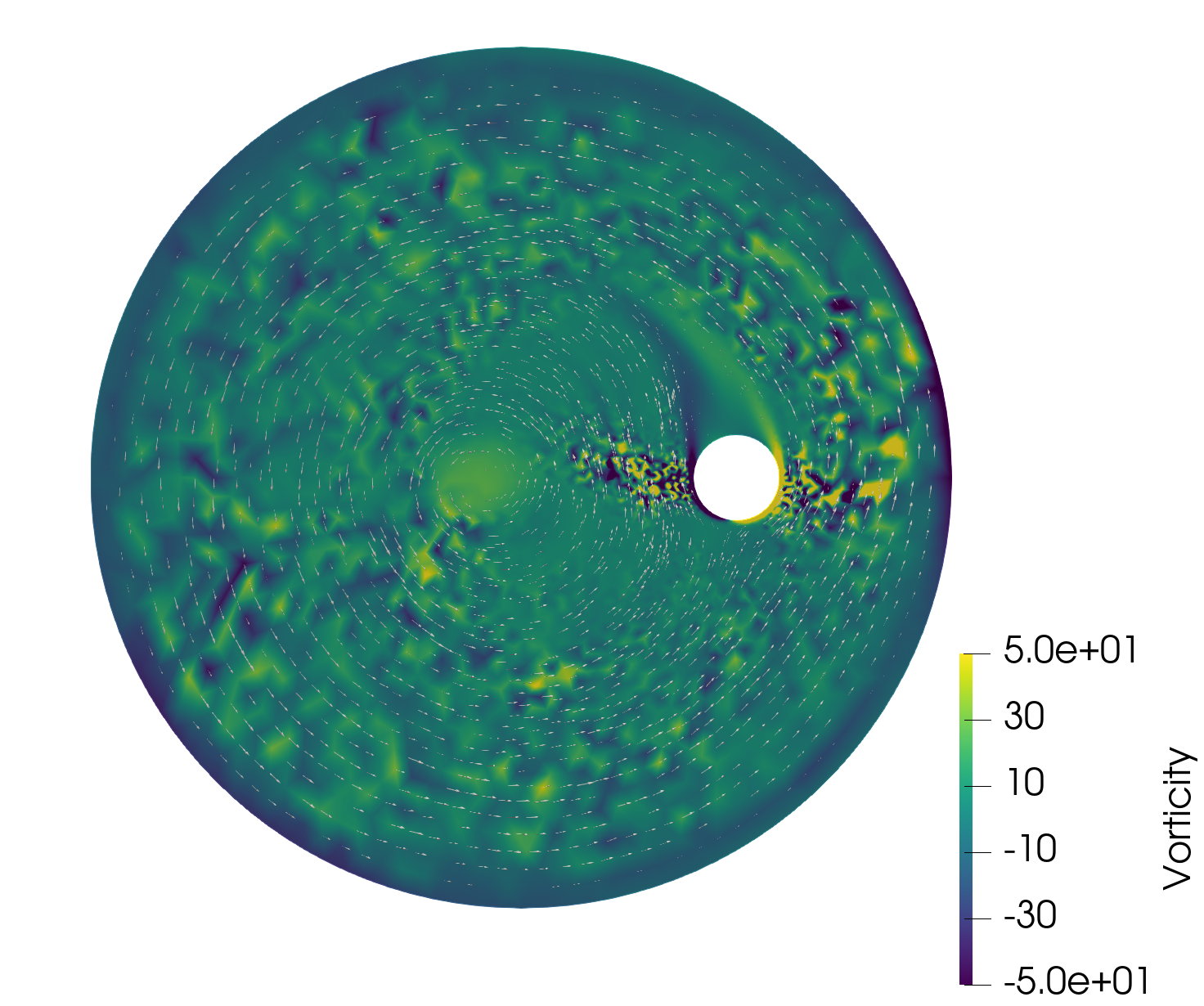}
    \end{subfigure}
    
    \vspace{0.3cm}
    \caption{Vorticity contour plots of the reference solution at times 0, 2, 4, and 6.}
    \label{fig:four_images_test3}
\end{figure}

\FloatBarrier

\section{Conclusions}

\label{sec6}

The main result of \cite{ref1} inspired many subsequent developments with the
aim of establishing an infinite predictability horizon under $H,\chi$
conditions. Interpretation of the $H,\chi$ arising from analysis in terms of
the Reynolds number in \cite{ref5} indicated that they are quite restrictive.
Reynolds number dependent (severe) restrictions are necessary to conclude
$T_{P}=\infty$ for the following heuristic reasons. Dissipation of errors
across all scales is necessary to prove $T_{P}=\infty$. Those in the
dissipation range (scales $\leq\mathcal{O}(\text{micro-scale})$) are
controlled by the viscous term. The nudging term controls scales between
$\mathcal{O}(\text{diam}(\Omega))$ and $\mathcal{O}(H)$. To have all error
growth controlled, so $T_{P}=\infty$, $H$ must thus overlap the dissipation range.

These considerations suggest seeking $T_{P}=\infty$ may be asking for too much
when data is sparse in space (or in time). Elaborating the contribution of
data assimilation to accuracy in a different direction (as herein) by analysis
of CDA with sparse data and moderate parameters, we believe, is an important
research direction. This report has shown that each instance data assimilated
strictly reduces error and thus strictly increases (the finite) predictability
horizon. Establishing stronger and more general results of this type is an
important open problem. It would also be interesting to establish a continuum
(in time) version of the result herein. One approach, inspired by e.g.
\cite{BGS18}, would be to compare the deviation from the NSE solution of the
two delay models%
\begin{align*}
\widetilde{u}_{t}+\widetilde{u}(t-\tau)\cdot\nabla\widetilde{u}(t-\tau
)-\nu\Delta\widetilde{u}+\nabla\widetilde{p} &  =f(x,t),\text{ }\nabla
\cdot\widetilde{u}=0\\
v_{t}+v(t-\tau)\cdot\nabla v(t-\tau)-\nu\Delta v+\nabla q-\chi I_{H}(u-v) &
=f(x,t)\text{, \ }\nabla\cdot v=0.
\end{align*}
Even though this seems very close to the time-discretized model studied
herein, standard estimates seem to fail, leaving an open problem.

\end{document}